\documentclass[reqno]{amsart}
\usepackage{graphicx}
\usepackage{amscd}
\usepackage{amsmath}
\usepackage{amsfonts}
\usepackage{amssymb}
\usepackage{cite}
\usepackage{xcolor}
\def\R {\mathbb{R}}
\def\C {\mathcal{C}}

\def\c{\mathfrak{C}}
\def\g{\mathfrak{G}}

\def\D {\mathbf{D}}
\newtheorem{proposition}{Proposition}[section]
\newtheorem{theorem}[proposition]{Theorem}
\newtheorem{corollary}{Corollary}[section]
\newtheorem{lemma}{Lemma}[section]
\theoremstyle{definition}
\newtheorem{remark}{Remark}[section]
\numberwithin{equation}{section}
\newtheorem{example}{Example}[section]

\begin{document}

\title[Regularized  Reconstruction of Scalar Parameters in  Subdiffusion]
{Regularized  Reconstruction of Scalar Parameters  in
\\ Subdiffusion with Memory via a Nonlocal Observation}

\author[A. Hulianytskyi, S. Pereverzyev, S.V. Siryk and N. Vasylyeva]
{Andrii Hulianytskyi, Sergei Pereverzyev, Sergii V. Siryk $\&$
Nataliya Vasylyeva}

\address{Taras Shevchenko National University
\newline\indent
str. Volodymyrska 64/13, 01601, Kyiv, Ukraine} \email[A.
Hulianytskyi]{andriyhul@gmail.com}

\address{Johann Radon Institute
\newline\indent
Altenbergstrasse 69, 4040 Linz, Austria} \email[S.
Pereverzyev]{sergei.pereverzyev@oeaw.ac.at}

\address{CONCEPT Lab, Istituto Italiano di Tecnologia
\newline\indent
Via Morego 30, 16163, Genova, Italy} \email[S.
Siryk]{accandar@gmail.com}

\address{Institute of Applied Mathematics and Mechanics of NAS of Ukraine
\newline\indent
G.Batyuka st.\ 19, 84100 Sloviansk, Ukraine;  and
\newline\indent
S.P. Timoshenko Institute of Mechanics of NAS of Ukraine
\newline\indent
Nesterov str.\ 3, 03057 Kyiv, Ukraine}
\email[N.Vasylyeva]{nataliy\underline{\ }v@yahoo.com}


\begin{abstract}
In the paper, we propose an analytical and numerical approach to
identify scalar parameters (coefficients, orders of fractional
derivatives) in the multi-term fractional differential operator in
time, $\mathbf{D}_t$. To this end, we analyze  inverse problems with
an additional nonlocal observation related to a linear subdiffusion
equation
$\mathbf{D}_{t}u-\mathcal{L}_{1}u-\mathcal{K}*\mathcal{L}_{2}u=g(x,t),$
where $\mathcal{L}_{i}$ are the second order elliptic operators with
time-dependent coefficients, $\mathcal{K}$ is a summable memory
kernel, and $g$ is an external force. Under certain assumptions on
the given data in the model, we derive explicit formulas for unknown
parameters. Moreover, we discuss the issues concerning to the
uniqueness and the stability in these inverse problems. At last, by
employing the Tikhonov regularization scheme with the
quasi-optimality approach, we give a computational algorithm to
recover the scalar parameters from a noisy discrete measurement and
demonstrate the effectiveness (in practice) of the proposed
technique via several numerical tests.

\end{abstract}

\maketitle

\section{Introduction}
\label{si}

\noindent Contemporary clinical treatments and medicines such as
cancer hypothermia, laser surgery, thermal ablation and thermal
disease diagnostic need comprehensive study (accounting memory
effect) of thermal phenomena, temperature behavior and the mass
transport of blood in biological tissues \cite{E1,GAV,MV,SM} (see
also references therein). Generalized heat transfer models are
described with usual wave or diffusion equations, while anomalous
diffusion processes acquire memory and nonlocal effects, which are
not easily captured within the framework of generalized heat
conduction theories exploiting partial differential equations with
derivatives of  integer order. Over the past few decades, fractional
variant of diffusion and wave equations (i.e. the equations with
fractional derivatives) provide a more realistic description of heat
transfer in materials and mass diffusion phenomena, which, in turn,
suggests an advanced mathematical approach to analysis and modeling
of various real-world phenomena. To derive these fractional
differential equations from physical laws, there are two different
ways. The first, so-called "microlevel" method, constructs on
modeling and passing to continuous limit. The second way is based on
conservative laws and specific constitutive relations with memory.
Indeed, following \cite{E1,PKLS} and appealing to the modified Green
and Naghdi III model with a phase-lag (see for details
\cite{GN1,RC}), we arrive at the constitutive relations (the
modified Fourier law and the classical energy equation) for the
temperature distribution in biological materials occupying a domain
$\Omega\subset \R^{n}$ ($n=1,2,3$),
\begin{equation}\label{i.1}
\begin{cases}
q(x,t+\tau)=-k_1\nabla\Theta(x,t)-k_2\nabla\frac{\partial
\Theta}{\partial
t}(x,t),\\
\varrho C_{\theta}\frac{\partial \Theta}{\partial t}(x,t)=-\text{div
}q(x,t)+Q,
\end{cases}
\end{equation}
where $q$ is the heat flux, $\Theta$ is the temperature variation of
each point $x\in\bar{\Omega}$ and time $t\in[0,T]$ from a uniform
temperature $\Theta_{0}$, the coefficients appearing in these
relations are positive constant material parameters, $Q$ is the
volumetric heat generated by metabolism and blood perfusion (see for
more details \cite[Section 2]{E1}). The meaning of a delay parameter
$\tau$ in the Fourier law in this model differs (generally) from the
commonly referred thermal relaxation time and may be of
comparatively large value, for example, $\tau$ changes from $16$s to
$30$s in meat product \cite{Ka}. Hence, to derive the governing
equation from \eqref{i.1} accounting memory effect, we first utilize
the fractional Taylor series \cite[Proposition 3.1]{J} and rewrite
the heat flux as
\begin{equation}\label{i.2}
q(x,t+\tau)=\sum_{i=0}^{K}\frac{\tau^{\mu_{i}}}{\Gamma(1+\mu_{i})}\mathbf{D}_{t}^{\mu_{i}}q(x,t)\quad\text{with}\quad
\mu_{i}=i\mu,\quad \mu\in(0,1),\quad \mu_{K}\leq 1.
\end{equation}
 Here, the symbol
$\D_{t}^{\mu_i}$ denotes the (regularized) left Caputo fractional
derivative of order $\mu_{i}\in(0,1]$ with respect to $t$ defined as
\[
\D_{t}^{\mu_{i}}q(\cdot,t)=\begin{cases}
\frac{1}{\Gamma(1-\mu_{i})}\frac{\partial}{\partial
t}\int_{0}^{t}\frac{[q(\cdot,s)-q(\cdot,0)]}{(t-s)^{\mu_{i}}}ds,\quad\mu_{i}\in(0,1),\\
\frac{\partial q}{\partial
t}(\cdot,t),\qquad\qquad\qquad\qquad\quad\mu_{i}=1,
\end{cases}
\]
where $\Gamma$ is the Euler Gamma-function. Plugging \eqref{i.2} to
the first equality  in \eqref{i.1} and then, under certain
assumptions on the function $\Theta$, performing straightforward
technical calculations, we end up with
\begin{equation}\label{i.3}
\varrho
C_{\theta}\sum_{i=0}^{K}\frac{\tau^{\mu_{i}}}{\Gamma(1+\mu_{i})}\mathbf{D}_{t}^{\mu_{i}}\Theta-k_2\Delta\Theta-k_1\int_{0}^{t}\Delta\Theta(x,s)ds
=\sum_{i=0}^{K}\frac{\tau^{\mu_{i}}}{\Gamma(1+\mu_{i})}I_{t}^{1-\mu_{i}}Q-k_2\Delta\Theta(x,0),
\end{equation}
where $I_{t}^{1-\mu_{i}}$ denotes the left fractional
Riemann-Liouville integral of order $1-\mu_i$ with respect to time
$t$.

\noindent Observing this equation, we remark that the orders of
fractional derivatives are unambiguously defined via parameters
$\mu$ and $K$, which are in general arbitrary. Indeed, if $\tau\geq
1$ and $\mu_{i}\in(0,1),$ then the coefficients at the fractional
derivatives in the series \eqref{i.2} can be of the same order of
smallness that arrives at the uncertainty in choice of number $K$ of
terms in the fractional Fourier series. \textit{Thus, in order to
complete the derivation of the equation modeling the heat conduction
in relevant biological environments, we have to  find the orders of
two fractional derivatives (e.g., $\mu_{K}$ and some $\mu_{i},$
$i\in\{1,2,...,K-1\}$) in \eqref{i.3} via additional data or
measurements.}

In connection with the latest, we mention that the similar problem
arises in the advanced model of oxygen transport through capillaries
\cite{STS,SR} (see also references therein), where the concentration
of oxygen $U=U(x,t)$ satisfies the two-term fractional subdiffusion
equation
\begin{equation}\label{i.4}
\mathbf{D}_{t}^{\nu_1}U-\tau_2\mathbf{D}_{t}^{\nu_2}U=\text{div}(a\nabla
U)-k-I_{t}^{\nu_1}(a_1\nabla U+a_2 U),\quad 0<\nu_2<\nu_1<1.
\end{equation}
Here, $\tau_2$ is the time lag in concentration of oxygen along the
capillary, $k$ is the rate of consumption per volume of tissue, and
$a,$ $a_i$ are the diffusion coefficients of oxygen. We notice that,
the term $\mathbf{D}_{t}^{\nu_1}U-\tau\mathbf{D}_{t}^{\nu_2}U$
describes  the net diffusion of oxygen to all tissues. In this
model, as in the previous one, the explicit values of $\nu_1,\nu_2$
are also not specified and should be again recovered  via solving
the corresponding inverse problems.

In this work, motivated by the above discussion, we focus on the
analytical and numerical investigation of inverse problems
concerning the identification of scalar parameters in two- and
multi-term fractional differential operator $\mathbf{D}_{t}$ (FDO)
in evolution equations.

 Let $\Omega$ be a bounded domain in $\R^{n}$ with boundary
$\partial\Omega$ belonging to $\C^{2+\alpha},$ $\alpha\in(0,1).$ For
any finite terminal positive $T,$ we set $
\Omega_{T}=\Omega\times(0,T)$ and $
\partial\Omega_{T}=\partial\Omega\times[0,T].$
Bearing in mind the model \eqref{i.4} and denoting the two-term
fractional differential operator with the time-depending
coefficients $\rho_{i}=\rho_{i}(t),$ $i=1,2,$ $\rho_{1}>0$,
\begin{equation}\label{2.2}
\D_{t}=
\begin{cases}
\rho_{1}\D_{t}^{\nu_{1}}-\rho_{2}\D_{t}^{\nu_{2}}\qquad \text{the I
type FDO},\\
\D_{t}^{\nu_{1}}\rho_{1}-\D_{t}^{\nu_{2}}\rho_{2}\qquad \text{the II
type FDO},
\end{cases}\quad  0<\nu_{2}<\nu_1<1,
\end{equation}
 we first consider  the inverse problem dealing with the linear
integro-differential equation with unknown function
$u=u(x,t):\Omega_{T}\to\R,$
\begin{equation}\label{2.1}
\mathbf{D}_{t}u-\mathcal{L}_{1}u-\mathcal{K}*\mathcal{L}_{2}u=g(x,t).
\end{equation}
Here $g$ is a given function, and $\mathcal{K}$ is a prescribed
memory kernel, the symbol $"*"$ stands for the usual time
convolution product
\[
(\eta_{1}*\eta_{2})(t)=\int_{0}^{t}\eta_{1}(t-s)\eta_{2}(s)ds.
\]
 As for $\mathcal{L}_{i},$ $i=1,2,$ they are linear
elliptic operators of the second order with time-depending
coefficients, which will be specified in Section \ref{s2}. The
equation \eqref{2.1} is supplemented with the initial condition and
the Neumann boundary condition
 \begin{equation}\label{2.3}
\begin{cases}
u(x,0)=u_{0}(x)\qquad\qquad\qquad\qquad\quad\text{in}\quad\bar{\Omega},\\
\mathcal{M}u+(1-d)\mathcal{K}*\mathcal{M}u=\varphi(x,t)\qquad\text{on}\quad
\partial\Omega_{T}
\end{cases}
 \end{equation}
 with $d=0$ or $1.$ The functions $u_{0}$ and $\varphi$ are
 specified below.
 Coming to the operator $\mathcal{M},$ it is the first order
 differential operator described in Section \ref{s2}.
 Finally, to complete the
 statement of  the inverse problem (IP), we introduce  the additional nonlocal
 measurement $\psi(t)$ having the form
 \begin{equation}\label{2.4}
\int_{\Omega}u(x,t)dx=\psi(t)
 \end{equation}
for small time $t\in[0,t^{*}],$ $t^{*}<\min\{1,T\}$.

 \noindent\textbf{Statement of the IP:} \textit{for the given
right-hand sides in \eqref{2.2}--\eqref{2.4}, coefficients in the
operators $\D_{t},$ $\mathcal{L}_{i},$ $\mathcal{M}$ and the memory
kernel $\mathcal{K},$ the inverse problem consists in the
identification of the triple $(\nu_1,\nu_2,u)$ such that
$\nu_{i}\in(0,1),$ $i=1,2,$ and $u$ solves the direct problem
\eqref{2.2}-\eqref{2.3} and satisfies the observation \eqref{2.4}
for small time.} \textit{Besides, in the case of $\rho_2$ being
unknown constant in $\mathbf{D}_{t}$, we also discuss IP related
with the reconstruction of $(\nu_1,\nu_2,\rho_2,u)$ by the
measurement \eqref{2.4}.}

\noindent Clearly, the IP in the latter case allows for the complete
identification of all parameters in the fractional operator in the
model \eqref{i.4} and, therefore, eliminates all uncertainties in
the model of oxygen distribution through capillaries based on the
approach utilizing fractional calculus.

In connection of the model \eqref{i.3}, here we also explore the IP
\eqref{2.1}, \eqref{2.3}, \eqref{2.4} with
 $M-$term fractional differential operator
($M>2$)
\begin{equation}\label{3.4*}
\mathbf{D}_{t}=
\begin{cases}
\sum\limits_{i=1}^{M}\rho_{i}(t)\mathbf{D}_{t}^{\nu_{i}}\qquad\quad\text{in
the case of the I type FDO},
\\
\sum\limits_{i=1}^{M}\mathbf{D}_{t}^{\nu_{i}}\rho_{i}(t)
\qquad\quad\text{in the case of the II type FDO},
\end{cases}
\, 0<\nu_{M}<...<\nu_{2}<\nu_1<1,
\end{equation}
 which concerns  the
recovery of $(\nu_1,\nu_i^{*},u)$ or
$(\nu_1,\nu_i^{*},\rho_i^{*},u)$ (if $\rho_{i^{*}}\equiv const.$),
$i^{*}\in \{2,3,...,M\}$. It is worth noting that, in contrast to
derivatives of  integer order,  the II type FDO has more complex
structure than $\mathbf{D}_{t}$ having form of the I type. Indeed,
in the case of fractional Caputo derivatives,   instead of the
well-known Leibniz rule, there is the representation
\[
\mathbf{D}_{t}^{\nu_{i}}(\rho_i(t)
u(x,t))=\rho_{i}(t)\mathbf{D}_{t}^{\nu_{i}}u(x,t)+u(x,0)\mathbf{D}_{t}^{\nu_{i}}\rho_{i}(t)
+\frac{\nu_{i}}{\Gamma(1-\nu_{i})}\int_{0}^{t}\frac{\rho_{i}(t)-\rho_{i}(s)}{(t-s)^{1+\nu_{i}}}[u(x,s)-u(x,0)]ds,
\]
if  $u$ and $\rho_{i}$ have the corresponding continuous fractional
derivatives (see for details \cite[Proposition 5.5]{SV}). Obviously,
even if
$\mathbf{D}_{t}^{\nu_{i}}u,\mathbf{D}_{t}^{\nu_i}\rho_i\in\C([0,T])$,
the last term in the right-hand side of this equality is a
convolution with a non-summable strongly singular kernel
$t^{-\nu_i-1}$ and, hence, overcoming this difficulty  requires
additional independent analytical study.

Lastly, the evolution equation \eqref{2.1} with the II type FDO can
be considered as a linearized version of fully nonlinear equations
similar to
\[
\sum_{i=1}^{M}\D_{t}^{\nu_{i}}(u\rho_{i}(t,u))
-\mathcal{L}_{1}u-\mathcal{K}*\mathcal{L}_{2}u=g(x,t),
\]
their special case  models heat transfer in multilayered materials
with thermosensitive features \cite{LMB}.

Inverse problems concerning with the recovery of the order to the
leading fractional derivative (i.e. $\nu_1$ in our notations) in the
one- or multi-term FDO like \eqref{3.4*} are studied in
\cite{HNWY,HPV,ICM,Ja,JK,KPSV,KPSV2,PSV1,SLJ,ZJY} (see also
references therein), where the different types of additional
observations are tested. It is worth noting that, there are two main
approaches in the above works. The first method is based on
obtaining  explicit formulas for $\nu_1$ in term of local or
nonlocal measurement (for small or large time)
\cite{HNWY,HPV,Ja,KPSV,KPSV2,Po,PSV1}. The second technique starting
from \cite{ICM} deals with the minimization of a certain functional
depending on both the solution of the corresponding direct problem
and given observation either for the terminal time $t=T$ or on the
whole time interval \cite{JK,SLJ,ZJY}. The one of main disadvantages
in the second approach concerns with huge numerical calculations
carried out in multidimensional domains and, besides, this method
needs (as a rule) not only a measurement but also
 all information on the coefficients and the right-hand sides in the direct problem,
 while a calculation by the explicit formula requires only the knowledge of the observation and has been done in a one-dimensional case.

 As for finding parameters in the I type FDO (see \eqref{3.4*}), the
 unique identification of $\rho_i,\nu_i,M$ in $\mathbf{D}_{t}$ in
 equation \eqref{2.1} with time-independent coefficients in the operators and $\mathcal{K}\equiv
 0$ is established in \cite{LLY,LY,LY1,JK}, where a local additional
 measurement either for small time, $t\in(0,t^{*})$, or on whole time interval $[0,T]$
 are considered. However, the key assumptions in these studies dictated by  techniques utilized (such as Laplace transformation, Fourier
 method) are time-independence of all coefficients in the equations
 and nonnegativity of all $\rho_i$, $i=2,...,M$. All this narrows the scope of  application.

 The stability in recovering $\nu_1$ by the local observation
 $u(x_0,t),$ $t\in(0,t^{*}),$ $x_{0}\in\Omega,$ is claimed in \cite{LHY} in the case of
 autonomous one-term fractional diffusion equations, and the similar
 results in the case of \eqref{2.1} with one- and multi-term
 $\mathbf{D}_{t}$ given by \eqref{3.4*} are obtained in
 \cite{KPSV,HPV}, where both local and nonlocal measurements for
 small time are considered.  At last, we mention that the influence
 of noisy observations on the reconstruction of $\nu_1$ in the case
 of one- and multi-term $\mathbf{D}_t$ are discussed in
 \cite{HNWY,KPSV,KPSV2,PSV,HPV}.

 Thus, having said that the picture is now pretty clear, there are still
some unexplored questions not addressed so far in the literature.
Namely, issues concerning to uniqueness, stability, impact of noisy
observation on the calculations of finding scalar parameters in
$\mathbf{D}_t$ having form either \eqref{2.2} or \eqref{3.4*} in the
nonautonomous subdiffusion equation \eqref{2.1} with memory terms
(i.e. $\mathcal{K}\neq 0$) are not studied. Moreover, in the most of
the previous published papers, finding these parameters via discrete
measurements blurred by a noise (that is more natural in real life)
are not discussed. By the authors' best knowledge, this in the case
of the recovery of $\nu_1$ is analyzed in \cite{ICM,KPSV,KPSV2,ZJY}
(for the one-term FDO) and in \cite{PSV1,HPV} (for $\mathbf{D}_{t}$
given by \eqref{3.4*}).

The present paper aims to provide some answers to the above
questions. The main achievements of this work can be summarized  in
the following points.

\noindent$\bullet$ Working in the framework of fractional H\"{o}lder
spaces, exploiting asymptotic behavior of $\psi(t)$ near $t=0$ and
analyzing some integral equations, we derive the explicit formulas
which allow us to identify successively unknown scalar parameters in
$\mathbf{D}_{t}$ under ceratin assumptions on the given data in
\eqref{2.2}-\eqref{3.4*}. In particular, in the case of FDO
appearing in the equation \eqref{i.4} with $a_1=0$, these formulas
are read as
\[
\nu_1=\underset{t\to 0}{\lim}\frac{\ln |\psi(t)-\psi(0)|}{\ln\,
t},\quad \nu_2=\nu_1-\log_{\lambda}\Big| \underset{t\to
0}{\lim}\frac{\mathcal{F}(\lambda t)}{\mathcal{F}(t)}
\Big|\quad\text{with}\quad \lambda\in(0,1),\quad
\tau_2=\frac{\mathcal{F}(t_0)}{\Big(\frac{t^{\nu_1-\nu_2-1}}{\Gamma(\nu_1-\nu_2)}*\mathbf{D}_{t}^{\nu_1}\psi\Big)(t_0)},
\]
where   $t_0\in(0,t^{*}]$ is chosen by the condition
$\mathcal{F}(t_0)\neq 0$, and
\[
\mathcal{F}(t)=\mathbf{D}_{t}^{\nu_1}\psi(t)+\int_{\Omega}k\,dx
+dI^{\nu_1}_{t}\Big(\int_{\partial\Omega}\varphi
dx\Big)-I_{t}^{\nu_1}(a_2\psi)-\int_{\partial\Omega}\varphi(x,t) dx.
\]

\noindent$\bullet$ The asymptotic behavior of $\psi(t)$ along with
the certain regularities of the given functions permit to prove the
unique identification and stability of unknown parameters via the
local measurement for short time interval.

\noindent$\bullet$ Assuming different behavior of an additive noise
at $t=0$, we obtain the error estimates of $|\nu_1-\nu_{1,\delta}|$
and $|\nu_{i^{*}}-\nu_{i^{*},\delta}|$, where $\nu_{1,\delta}$ and
$\nu_{i^{*},\delta}$ are calculated via the noisy measurement
$\psi_{\delta}$. Finally, using Tikhonov regularization scheme with
the quasi-optimality approach, we propose the computational
algorithm to reconstruct unknown parameters $\nu_1,\nu_{i^{*}},
\rho_{i^{*}}$ via the discrete noisy observation $\psi_{\delta}$.
The effectiveness of this computational approach is justified via
several numerical tests.

In fact,  this study offers  a new analytical and numerical approach
for reconstructing unknown parameters in FDO like \eqref{3.4*} or
\eqref{2.2}, which can be used in practice, for example, in the
models described by \eqref{i.3} and \eqref{i.4}. It is worth noting
that in all arguments in this paper, we do not require either the
time independence of the coefficients in \eqref{2.1} or the
positivity of $\rho_i$, $i=2,...,M$,  widespread in the previous
 literature.

\textbf{The paper is organized as follow:} In the next section, we
introduce the functional spaces and notations, and describe the main
results concerning the unique reconstruction of triples
$(\nu_1,\nu_2,u)$ and $(\nu_{1},\nu_{i^{*}},u)$ which are stated in
Theorems \ref{t3.1} and \ref{t3.2}, respectively. Moreover, in
Section \ref{s2}, we establish stability bounds in IPs (Lemma
\ref{l3.2}). The verification of Theorems \ref{t3.1}--\ref{t3.2} and
Lemma \ref{l3.2} is carried out in Sections \ref{s3} and \ref{s4},
respectively. The one-valued reconstruction of
$(\nu_1,\nu_2,\rho_2,u)$ and $(\nu_1,\nu_i^{*},\rho_i^{*},u)$ stated
in Theorems \ref{t5.1} and \ref{t5.2} is discussed in Section
\ref{s5}. The influence of noise on the computations of $\nu_1,$
$\nu_2$ or $\nu_{i^{*}}$ is analyzed in Section \ref{s7}. Finally,
under the assumptions of discrete noise measurement $\psi(t)$,
 the description of the computational
algorithm for regularized recovery of $(\nu_1,\nu_2,\rho_2)$ and
$(\nu_1,\nu_i^{*},\rho_i^{*})$ are described in Section \ref{s8}.
Besides, the effectiveness of this method is demonstrated via
numerical tests  in this section.
\section{Main Results: Reconstruction of the triples $(\nu_1,\nu_{2},u)$ and $(\nu_1,\nu_{i^{*}},u)$}
\label{s2}

\subsection{Functional setting} We study problems
\eqref{2.1}--\eqref{3.4*} in the fractional H\"{o}lder spaces
$\C^{l+\alpha,\frac{l+\alpha}{2}\mu}(\bar{\Omega}_{T})$ (see for
detail \cite[Section 2]{KPV3}), $l=0,1,2,$ $\alpha,\mu\in(0,1),$
endowed with the norm
\[
\|u\|_{\C^{l+\alpha,\frac{l+\alpha}{2}\mu}(\bar{\Omega}_{T})}=
\begin{cases}
\|u\|_{\C([0,T],\C^{l+\alpha}(\bar{\Omega}))}+\sum\limits_{|j|=0}^{l}\langle
D_{x}^{j}u\rangle_{t,\Omega_{T}}^{(\frac{l+\alpha-|j|}{2}\mu)},\qquad\qquad\qquad\qquad\qquad
l=0,1,\\
\|u\|_{\C([0,T],\C^{2+\alpha}(\bar{\Omega}))}+\|\D_{t}^{\mu}u\|_{\C^{\alpha,\frac{\mu\alpha}{2}}(\bar{\Omega}_{T})}+\sum\limits_{|j|=1}^{2}\langle
D_{x}^{j}u\rangle_{t,\Omega_{T}}^{(\frac{2+\alpha-|j|}{2}\mu)},\qquad
l=2,
\end{cases}
\]
where $\langle \cdot\rangle_{t,\Omega_{T}}^{(\alpha)}$ and $\langle
\cdot\rangle_{x,\Omega_{T}}^{(\alpha)}$ stand for the standard
H\"{o}lder seminorms of a function  with respect to time and space
variable, respectively.

In this article, we will also utilize the Hilbert space
$L_{w}^{2}(t_1,t_2)$ of real-valued square integrable functions with
a positive weight $w=w(t)$ on $(t_1,t_2).$ The inner product and the
norm in $L_{w}^{2}(t_1,t_2)$ are defined as
\[
\langle u,v\rangle_{L_{w}^{2}}=\int_{t_1}^{t_2}w(t)v(t)u(t)dt\qquad
\text{and}\quad\|u\|_{L_{w}^{2}(t_1,t_2)}=\sqrt{\int_{t_1}^{t_2}w(t)u^{2}(t)dt}.
\]

\subsection{General assumptions in the model}\label{s2.2}

\begin{description}
\item[h1. Conditions on the operators]
The operators $\mathcal{L}_{i}$ and $\mathcal{M}$  are defined as
\[
\mathcal{L}_{1} = \sum_{ij=1}^{n}\frac{\partial}{\partial x_{i}}
b_{ij}(x,t)\frac{\partial}{\partial x_{j}} +a_{0}(t),\quad
\mathcal{L}_{2} = \sum_{ij=1}^{n}\frac{\partial}{\partial x_{i}}
b_{ij}(x,t)\frac{\partial}{\partial x_{j}} +b_{0}(t),\quad
\mathcal{M}=
-\sum_{ij=1}^{n}b_{ij}(x,t)N_{i}\frac{\partial}{\partial x_{j}}
\]
with $\mathbf{N}=\{N_{1},...,N_{n}\}$ being the unit outward normal
vector to $\Omega$.

\noindent There exist constants $\varrho_{2}>\varrho_1>0,$  such
that
     \begin{equation*}
\varrho_{1}|\xi|^{2}
\leq\sum_{ij=1}^{n}b_{ij}(x,t)\xi_{i}\xi_{j}\leq\varrho_{2}|\xi|^{2}
\quad
   \text{for any}\quad (x,t,\xi)\in\bar{\Omega}_{T}\times \mathbb{R}^{n}.
   \end{equation*}
       \item[h2. Conditions on the FDO] We require that
    $\nu_{1}\in(0,1)$ and the remaining  $\nu_i\in(0,
\nu_1\frac{2-\alpha}{2}).$
 Moreover, there is  a positive constant $\varrho_{3}$ such that
 \[
\rho_{1}(t)\geq \varrho_{3}>0
 \quad \text{for all}\quad t\in[0,T].
 \]
    \item[h3. Regularity of  the coefficients in \eqref{2.1} and \eqref{2.2}] For
    $\nu\in(1,1+\alpha/2),$
    there hold
    \begin{equation*}
  a_{0},
  b_{0} \in
    \C^{\frac{\alpha}{2}}([0,T]),\quad
b_{ij}\in
 \C^{1+\alpha,\frac{1+\alpha}{2}}(\bar{\Omega}_{T}),\, i,j=1,\ldots,n,\quad
 \rho_{k}\in\C^{\nu}([0,T]), \, k=1,\ldots,M.
    \end{equation*}
             \item[h4. Smoothness of the given functions]
       \[
       \mathcal{K}\in L_{1}(0,T),\quad
 \varphi\in\C^{1+\alpha,\frac{1+\alpha}{2}}(\partial\Omega_{T}),\quad
u_{0}\in C^{2+\alpha}(\bar{\Omega}), \quad
g\in\C^{\alpha,\frac{\alpha}{2}}(\bar{\Omega}_{T}).
    \]
    \item[h5. Condition on the additional measurement]
    We assume that $\psi\in\C([0,t^{*}])$ has $M$ fractional Caputo
    derivatives of order  less than $1$, and all these
    derivatives are H\"{o}lder continuous.
    \item[h6. Compatibility conditions] For every
         $x\in\partial\Omega$ at the initial time
     $t=0$, there holds
       \begin{equation*}
\mathcal{M}u_{0}(x)|_{t=0}=\varphi(x,0).
    \end{equation*}
    \end{description}
\subsection{Statement of the main results}\label{s2.3}
Now we are in the position to state our main results. The first of
them concerns to the reconstruction of the triple $(\nu_1,\nu_2,u)$,
the latter  means that we should put $M=2$ in the assumptions above.
Assuming $\Omega=(\mathfrak{l}_{1},\mathfrak{l}_{2})$ in the
one-dimensional case, we first introduce the functions:
\begin{align}\label{3.1}\notag
\mathcal{I}(t)&=\begin{cases}
\int_{\partial\Omega}\varphi(x,t)dx\quad\qquad\text{ if}\qquad n\geq 2,\\
\varphi(\mathfrak{l}_{2},t)-\varphi(\mathfrak{l}_{1},t)\qquad
\text{if}\quad n=1,
\end{cases}\\\notag
\c(t)&=\int_{\Omega}g(x,t)dx-d(\mathcal{K}*\mathcal{I})(t)+a_{0}(t)\psi(t)+(\mathcal{K}*b_{0}\psi)(t)-\mathcal{I}(t),\quad
\c_{0}=\c(0),
\\
\mathcal{F}(t)&=\begin{cases}
\rho_{2}^{-1}(t)[\rho_{1}(t)\D_{t}^{\nu_{1}}\psi(t)-\c(t)]\qquad
\text{in the case of the I type FDO},\\
\D_{t}^{\nu_{1}}(\rho_{1}(t)\psi(t))-\c(t)\qquad\qquad \text{ in the
case of the II type FDO}.
\end{cases}
\end{align}
\begin{theorem}\label{t3.1}
Let positive  $T$ be arbitrary but finite, $\c_{0}\neq 0$ and
$\rho_{2}(t)\neq 0$ for any $t\in[0,t^{*}]$. Under assumptions
h1-h6, the inverse problem \eqref{2.2}-\eqref{2.4} has a unique
solution $(\nu_{1},\nu_{2},u)$. Besides, $\nu_{1}$ and $\nu_{2}$ are
successively computed via formulas:
\begin{equation}\label{3.2*}
\nu_{1}=
\begin{cases}
\underset{t\to
0}{\lim}\frac{\ln|\psi(t)-\int_{\Omega}u_{0}(x)dx|}{\ln
t}\qquad\qquad\qquad\quad\text{in the case of the I type FDO},
\\
\underset{t\to
0}{\lim}\frac{\ln|\rho_{1}(t)\psi(t)-\rho_{1}(0)\int_{\Omega}u_{0}(x)dx|}{\ln
t}\qquad\quad\text{in the case of the II type FDO},
\end{cases}
\end{equation}
and
\begin{equation}\label{3.3*}
\nu_{2}=\nu_{1}-\log_{\lambda}\bigg|\underset{t\to
0}{\lim}\frac{\mathcal{F}(\lambda t)}{\mathcal{F}(t)}\bigg|
\end{equation}
with  $\lambda\in(0,1);$ while the function $u$ is a unique solution
of \eqref{2.2}-\eqref{2.3}, which  has the regularity
\[
u\in\C^{2+\alpha,\frac{2+\alpha}{2}\nu_{1}}(\bar{\Omega}_{T})\quad\text{and}\quad
\D_{t}^{\nu_2}u\in\C^{\alpha,\frac{\alpha\nu_{1}}{2}}(\bar{\Omega}_{T}).
\]
\end{theorem}
The next claim deals with the identification of the triple
$(\nu_1,\nu_{i^{*}},u)$ in the case of $M$-term fractional
differential operator \eqref{3.4*} (i.e. $M>2$).
\begin{theorem}\label{t3.2}
Let $M>2$, any positive $T$ be finite and  assumptions h1-h6 hold.
If $\c_{0}\neq 0,$ and $\rho_{i^{*}}(t)\neq 0$ for all
$t\in[0,t^{*}]$, then IP \eqref{2.1}--\eqref{3.4*} admits a unique
solution $(\nu_{1},\nu_{i^{*}},u),$ where $\nu_1$ is calculated by
\eqref{3.2*}, while $\nu_{i^{*}}$ is computed via \eqref{3.3*} with
\[
\mathcal{F}(t)=
\begin{cases}
\rho_{i^{*}}^{-1}(t)\big[\c(t)-\sum\limits_{j=1,j\neq
i^{*}}^{M}\rho_{j}(t)\D_{t}^{\nu_{j}}\psi(t)\big]\qquad
\text{in the case of the I type FDO},\\
\c(t)-\sum\limits_{j=1,j\neq
i^{*}}^{M}\D_{t}^{\nu_{j}}(\rho_{j}(t)\psi(t))\qquad\qquad \text{in
the case of the II type FDO}.
\end{cases}
\]
Besides, the function
$u\in\C^{2+\alpha,\frac{2+\alpha}{2}\nu_{1}}(\bar{\Omega}_{T})$
solves the direct problem \eqref{2.1}, \eqref{3.4*}, \eqref{2.3}
and, besides,
$\D_{t}^{\nu_2}u\in\C^{\alpha,\frac{\alpha\nu_{1}}{2}}(\bar{\Omega}_{T})$.
\end{theorem}
\noindent The following assertion is related to the dependence of a
solution $u$ on the orders $\nu_{i}$. It is worth noting that, this
issue in the case of $\D_{t}$ being a single-term fractional
differential operator (i.e. $\rho_{i}\equiv 0,$ $i=2,\ldots M$) is
discussed in \cite[Lemma 1]{KPSV}. Actually, this result can be
easily extended (with slightly modifications in the arguments) to
the case of $\D_{t}$ having the form either \eqref{2.2} or
\eqref{3.4*}. Therefore, the stability only in the case of
$\nu_{i}$, $i\neq 1,$ is still unexplored question.
\begin{lemma}\label{l3.2}
Let $M\geq 2$, $\alpha,\nu_{1}\in(0,1)$,  and $i\in\{2,\ldots,M\}$
be fixed, and let
$0<\beta_{1,i}<\beta_{2,i}<\frac{2-\alpha}{2}\nu_{1}$. We assume
that assumptions of either Theorem \ref{t3.1} if $M=2$ or Theorem
\ref{t3.2} if $M>2$ hold. If $u_{1,i}=u_{1,i}(x,t)$ and
$u_{2,i}=u_{2,i}(x,t)$ solve \eqref{2.1}, \eqref{2.3}   where
$\mathbf{D}_{t}$ given by either \eqref{2.2} ($M=2$) or \eqref{3.4*}
($M>2$) with $\nu_{i}$ being replaced by $\beta_{1,i}$ and
$\beta_{2,i},$ respectively, then there is the estimate
\[
\|u_{1,i}-u_{2,i}\|_{\C^{2+\alpha,\frac{2+\alpha}{2}\nu_{1}}(\bar{\Omega}_{T})}\leq
C(\beta_{2,i}-\beta_{1,i})[\|u_{0}\|_{\C^{2+\alpha}(\bar{\Omega})}+\|g\|_{\C^{\alpha,\frac{\nu_{1}\alpha}{2}}(\bar{\Omega}_{T})}
+
\|\varphi\|_{\C^{1+\alpha,\frac{1+\alpha}{2}\nu_{1}}(\partial\Omega_{T})}
]
\]
with the positive value $C$ being independent of
$(\beta_{2,i}-\beta_{1,i})$.
\end{lemma}
\noindent Proofs of Theorem \ref{t3.1} and Lemma \ref{l3.2} are
given in Sections \ref{s4} and \ref{s5}, respectively. As for
Theorem \ref{t3.2}, its verification is carried out with slightly
modification in the arguments of Section \ref{s4}, and therefore we
omit it here.

\section{Proof of Theorem \ref{t3.1}}\label{s3}

\noindent To prove this claim, we will incorporate the technique
consisting in two main steps. At  the first stage, we focus on the
existence of the triple $(\nu_{1},\nu_{2},u)$. Concerning the orders
$\nu_{1}$ and $\nu_{2}$, we need to validate formulas \eqref{3.2*}
and \eqref{3.3*}. We notice that formula \eqref{3.2*} recovering
$\nu_1$ has been proved in \cite[Theorem 2.1]{HPV}. Thus, here we
are just left to verify  \eqref{3.3*} to find  $\nu_{2}$ and, then,
substituting the searched orders to \eqref{2.1}, to prove the
 classical global solvability of the direct problem
\eqref{2.2}-\eqref{2.3}. As for the  formula of $\nu_2$, using the
reconstructed value $\nu_{1}$ and integrating the equation
\eqref{2.1} over $\Omega$, we reduce relations
\eqref{2.2}-\eqref{2.4} to the equality with a weaker kernel
\begin{equation*}\label{3.1*}
(\omega_{\nu_1-\nu_2}*v) (t)=v(t)+\mathcal{F}_{1}(t)\qquad\text{for
each}\quad t\in[0,t^{*}],
\end{equation*}
where
$\omega_{\theta}=\omega_{\theta}(t)=\frac{t^{\theta-1}}{\Gamma(\theta)},$
$\theta\in(0,1)$, and the function $v$ is defined via the term
$\D_{t}^{\nu_1}(\rho_{2}\psi)$ or $\D_{t}^{\nu_{1}}\psi$, while the
function $\mathcal{F}_{1}(t)$ is represented with a linear
combination of $\D_{t}^{\nu_{1}}\psi$ and $\mathcal{F}(t)$. After
that, under additional assumptions on the given functions, we show
that $\nu_2$ given by \eqref{3.3*} satisfies \eqref{3.1*}. Finally,
exploiting the searched orders $\nu_1$ and $\nu_2$ in
\eqref{2.2}-\eqref{2.3}, we solve the direct problem to find the
unknown function $u$. On this route, we utilize \cite[Theorem 4.1,
Remark 4.4]{PSV} which provide the existence of $u$ in the
corresponding fractional H\"{o}lder classes. As a result, we
reconstruct the triple $(\nu_1,\nu_2,u)$ which solves the IP
\eqref{2.2}-\eqref{2.4}. The second stage in the arguments concerns
the uniqueness of a solution to \eqref{2.2}-\eqref{2.4}. To this
end, we appeal to the arguments by contradiction. Namely, assuming
two different solutions of IP (with the same given functions, the
coefficients and the measurement), we will examine that this IP
admits no more than one solution if assumptions of Theorem
\ref{t3.1} hold.

\subsection{Auxiliary results}\label{s3.1}

Here, we establish  technical results playing a key role in the
verification of formula \eqref{3.3*}.
\begin{proposition}\label{p3.1}
Let a continuous function $w=w(t):[0,T]\to \R$ have  continuous
Caputo fractional  derivatives in time of orders $\mu_{1}$ and
$\mu_{2},$ $0<\mu_{2}<\mu_{1}<1$. Then there is the representation
\[
\D_{t}^{\mu_{2}}w(t)=(\omega_{\mu_{1}-\mu_{2}}*\D_{t}^{\mu_{1}}w)(t)
\]
for all $t\in[0,T]$.
\end{proposition}
\begin{proof}
Setting $w_{0}=w(0)$ and appealing to the definition of the
fractional Caputo derivative and \cite[Proposition 4.2]{KPV3}, we
conclude that
\[
\D_{t}^{\mu_{2}}w= \frac{\partial}{\partial
t}(\omega_{1-\mu_2}*[w-w_{0}])
 =\frac{\partial}{\partial
t}(\omega_{1-\mu_1}*\omega_{\mu_{1}-\mu_{2}}*[w-w_{0}])
\]
for all $t\in[0,T].$ After that, collecting the definition of the
Riemann-Liouwille fractional derivative $\partial_{t}^{\theta}$ with
\cite[Lemma 2.10]{KST} arrives at the relations
\begin{align*}
\D_{t}^{\mu_{2}}w(t)&=\partial_{t}^{\mu_{1}}(\omega_{\mu_{1}-\mu_{2}}*[w-w_{0}])(t)
=
\Big(\omega_{\mu_1-\mu_2}*\D_{t}^{\mu_{1}}w\Big)(t)
\\
& +\omega_{\mu_1-\mu_2}(t)\underset{z\to
0}{\lim}(\omega_{1-\mu_{1}}*(w-w_{0}))(z).
\end{align*}
Thanks to the continuity of $w(t)$, the second term in the last
equality vanishes. That completes the verification of this claim.
\end{proof}

Here, for reader's convenience, we recall results, which subsume
Lemma 3.2 and Remark 3.1 in \cite{HPV} and  concern with finding the
order of a weaker singularity in a convolution.
\begin{lemma}\label{l.convolution}
Let arbitrary $T>0$, and $f=f(t):[0,T]\to\R$ and $k=k(t):[0,T]\to\R$
be bounded and continuous functions satisfying the relations
\begin{equation*}
f(0)\neq 0 \quad\text{and}\quad  k(0)\neq 0.
\end{equation*}
Then for  $\lambda\in(0,1)$ and
\[
\mathcal{G}(t)=\int_{0}^{t}(t-\tau)^{\gamma^{*}-1}k(t-\tau)f(\tau)d\tau\quad\text{with}\quad
\gamma^{*}\in(0,1),
\]
the following equalities hold:
\begin{equation*}
\underset{t\to 0}{\lim}\frac{\mathcal{G}(\lambda
t)}{\mathcal{G}(t)}=\lambda^{\gamma^{*}}\quad\text{and}\quad
\gamma^{*}=\log_{\lambda}\bigg|\underset{t\to
0}{\lim}\frac{\mathcal{G}(\lambda t)}{\mathcal{G}(t)}\bigg|.
\end{equation*}
\end{lemma}
Next, we focus on the extension of this result to a weaker singular
kernel $\omega_{\theta}(t)$ and  given functions $a=a(t),$ $v=v(t)$
and $f_{0}=f_{0}(t)$ defined in $[0,T]$, which are related via the
equality
\begin{equation}\label{3.2}
(\omega_{\theta}*v)(t)=\mathcal{F}_{0}(t)\qquad\text{for}\quad
t\in[0,T],\quad \text{where}\quad
\mathcal{F}_0(t)=a(t)v(t)+f_{0}(t).
\end{equation}
\begin{lemma}\label{l4.1}
Let $\theta\in(0,1)$ and the functions $v$ and $\mathcal{F}_{0}$ be
continuous in $[0,T]$ and
\begin{equation}\label{3.3}
\frac{v(0)}{n^{*}}+\mathcal{F}_{0}(0)\neq 0
\end{equation}
for some $n^{*}\in\mathbb{N}.$ If equality \eqref{3.3} holds for
each $t\in[0,T]$, then
\begin{equation}\label{b.1}
\theta=\log_{\lambda}\bigg|\underset{t\to
0}{\lim}\frac{\mathcal{F}_{0}(\lambda t)}{\mathcal{F}_{0}(t)}\bigg|
\end{equation}
for each $\lambda\in(0,1)$.
\end{lemma}
\begin{proof}
This claim is a consequence of Lemma \ref{l.convolution} and the
properties of the Mittag-Leffler function
\[
E_{\theta}(z)=\sum_{k=0}^{\infty}\frac{z^{k}}{\Gamma(\theta k+1)}.
\]
Namely, \cite[Proposition 4.1]{KPSV5} establishes that the function
$s_{n,\theta}:=s_{n,\theta}(t)=E_{\theta}(-nt^{\theta}),$
$n\in\mathbb{N},$ $\theta\in(0,1),$ solves the scalar-valued Voltera
equation
\begin{equation}\label{3.5}
s_{n,\theta}(t)+n(\omega_{\theta}*s_{n,\theta})(t)=1,\quad t\geq 0.
\end{equation}
Taking convolution \eqref{3.2} with $s_{n,\theta}$, we arrive at the
equality
\[
(s_{n,\theta}*\omega_{\theta}*v)(t)=(s_{n,\theta}*\mathcal{F}_0)(t),
\]
and then exploiting \eqref{3.5} and the commutative and associative
properties of convolution, we end up with
\[
\bigg(\frac{1-s_{n,\theta}}{n}*v\bigg)(t)=(s_{n,\theta}*\mathcal{F}_0)(t)
\]
for each $t\in[0,T]$ and any fixed  $n\in\mathbb{N}.$

\noindent Finally, setting
\[
\mathcal{U}=\mathcal{U}(n,t)= \frac{v(t)}{n}+\mathcal{F}_{0}(t),
\]
we rewrite the last equality in the form
\begin{equation}\label{3.6}
(s_{n,\theta}*\mathcal{U})(t)=(\frac{1}{n}*v)(t)
\end{equation}
 for each $t\geq 0$  and any fixed $n\in\mathbb{N}.$

\noindent Thanks to condition \eqref{3.3} and  the regularity of $v$
and $\mathcal{F}_{0}$, we easily deduce that
\begin{equation}\label{3.7}
\mathcal{U}(n^{*},0)\neq 0\quad\text{and}\quad
\mathcal{U}(n,t)\in\C([0,T])
\end{equation}
for each fixed $n\in\mathbb{N}.$ Collecting this fact with the
straightforward calculations leads to the following equalities:
\begin{align*}
\frac{d}{dt}(n^{-1}*v)(t)&=n^{-1}v(t),\\
\frac{d}{dt}(s_{n,\theta}*\mathcal{U})(t)&=s_{n,\theta}(0)\mathcal{U}(n,t)
+
\int_{0}^{t}\mathcal{U}(n,\tau)\frac{d}{d(t-\tau)}s_{n,\theta}(t-\tau)d\tau,
\end{align*}
where due to \cite[Proposition 4.1]{KPSV5}
\[
s_{n,\theta}(0)=1\quad\text{and}\quad \frac{d}{dt}s_{n,\theta}(t)=-n
t^{\theta-1}E_{\theta,\theta}(-nt^{\theta}).
\]
Here $E_{\theta_{1},\theta_{2}}(z)$ is the two-parametric
Mittag-Leffler function defined as
\begin{equation}\label{m.1}
E_{\theta_{1},\theta_{2}}(z)=\sum_{k=0}^{+\infty}\frac{z^{k}}{\Gamma(\theta_{1}k+\theta_{2})},\quad
\theta_{1},\theta_{2}>0.
\end{equation}
At this point, differentiating  \eqref{3.6} with respect to  $t$ and
utilizing the relations above, we deduce the equality
\begin{equation}\label{3.8}
G(t,n)=\mathcal{F}_{0}(t)
\end{equation}
for any fixed $n\in\mathbb{N}$ and each  $t\in[0,T]$, where we set
\[
G(t,n)=\int_{0}^{t}(t-\tau)^{\theta-1}n
E_{\theta,\theta}(-n(t-\tau)^{\theta})\mathcal{U}(n,\tau)d\tau.
\]
In fine,  properties of $\mathcal{U}(n,\tau)$ (see \eqref{3.7}) and
$E_{\theta_{1},\theta_{2}}(z)$ allow us to  apply Lemma
\ref{l.convolution} to $\mathcal{G}(t)=G(t,n^{*})$. Namely,
selecting
\[
\mathcal{G}(t)=G(t,n^{*}),\quad
k(t)=n^{*}E_{\theta,\theta}(-n^{*}t^{\theta}),\quad
f(t)=\mathcal{U}(n^{*},t),\quad \gamma^{*}=\theta,
\]
we obtain the equality
\[
\lambda^{\theta}=\underset{t\to 0}{\lim}\frac{G(\lambda
t,n^{*})}{G(t,n^{*})}
\]
 with arbitrary $\lambda\in(0,1)$.
In fine, appealing to equality  \eqref{3.8} with $n=n^{*}$,   we
complete the verification
 of Lemma \ref{l4.1}.
 \end{proof}
\begin{remark}\label{r4.1}
It is apparent that:

\noindent(i) If $v(0)=0$ and $f_{0}(0)\neq 0,$ then condition
\eqref{3.3} holds for any $n\in\mathbb{N}.$

\noindent(ii) If $v(0)\neq 0$ and $f_{0}(0)= 0,$ then condition
\eqref{3.3} holds for any $n\in\mathbb{N}$ satisfying the inequality
\[
n^{-1}\neq -a(0).
\]
Clearly, arbitrary positive $a(0)$  provides the fulfillment of the
last inequality (and consequently, \eqref{3.3}) with any
$n\in\mathbb{N}$.

\noindent(iii) If $v(0)$ does not vanish and $f_{0}(0)\neq 0$, then
condition \eqref{3.3} holds for any $n\in\mathbb{N}$ satisfying the
inequality
\[
n^{-1}\neq -\frac{\mathcal{F}_{0}(0)}{v(0)}.
\]
Obviously, if the value $\frac{\mathcal{F}_{0}(0)}{v(0)}$ is
nonnegative, then \eqref{3.3} is fulfilled for any $n\in\mathbb{N}.$
Otherwise, that is in the case of  the negative
$\frac{\mathcal{F}_{0}(0)}{v(0)}$, \eqref{3.3} holds for any integer
positive $n$ solving the inequality
\[
n>\bigg|\frac{v(0)}{\mathcal{F}_{0}(0)}\bigg|.
\]
For example,
$n=1+\Big[\Big|\frac{v(0)}{\mathcal{F}_{0}(0)}\Big|\Big]$, where the
symbol  $[\cdot]$ stands for the integer part of a number.
\end{remark}

We complete this subsection with the asymptotic representation
established in \cite[Lemma 4.1]{PSV1} and reported here below in a
particular form tailored for our goals. To this end, for given
functions $v=v(t)$ and $r_{i}=r_{i}(t),$ $i=1,2,$ and the parameters
$\theta\in(0,1)$ and $\mu_{i}:$ $0<\mu_{2}<\mu_{1}<1,$ we set
\begin{align*}
\D_{t}^{I}v&=r_{1}(t)\D_{t}^{\mu_{1}}v(t)-r_{2}(t)\D_{t}^{\mu_{2}}v(t)\quad\text{and}\quad
\D_{t}^{II}v=\D_{t}^{\mu_{1}}(r_{1}(t)v(t))-\D_{t}^{\mu_{2}}(r_{2}(t)v(t)),\\
J_{\theta}(v,t)&=\int_{0}^{t}(t-\tau)^{\theta-1}[\D_{\tau}^{\theta}v(\tau)-\D_{\tau}^{\theta}v(0)]d\tau.
\end{align*}
\begin{lemma}\label{l4.2}
Let positive $T$ be any but fixed, $0<\mu_{2}<\mu_{1}<1,$ and
$r_1,r_2\in\C^{1+\alpha^{*}}([0,T])$ with $\alpha^{*}\in(0,1),$ and,
besides,  $r_1$ be a positive function. We assume that a continuous
function $v=v(t):[0,T]\to\R$ has continuous derivatives
$\D_{t}^{\mu_{1}}v,$ $\D_{t}^{\mu_{2}}v$ in $[0,T]$. Then for each
$t\in[0,T]$ the following representations hold:
\begin{align*}
&[v(t)-v(0)][r_{1}(0)\Gamma(1+\mu_{1})-r_{2}(0)t^{\mu_{1}-\mu_{2}}\Gamma(1+\mu_{2})]\\&=
t^{\mu_{1}}\D_{t}^{I}v(0)+\mu_{1}r_{1}(0)J_{\mu_{1}}(v,t)-\mu_{2}r_{2}(0)t^{\mu_{1}-\mu_{2}}J_{\mu_{2}}(v,t);\\
&[r_{1}(t)v(t)-r_{1}(0)v(0)]\Gamma(1+\mu_{1})-[r_{2}(t)v(t)-r_{2}(0)v(0)]t^{\mu_{1}-\mu_{2}}\Gamma(1+\mu_{2})\\&=
t^{\mu_{1}}\D_{t}^{II}v(0)+\mu_{1}J_{\mu_{1}}(r_1v,t)-\mu_{2}t^{\mu_{1}-\mu_{2}}J_{\mu_{2}}(r_2v,t).
\end{align*}
\end{lemma}

\subsection{Solvability of \eqref{2.2}-\eqref{2.4}}\label{s3.2}

First, we recall the results concerning the solvability and
regularity of the direct problem \eqref{2.2}-\eqref{2.3}, which
subsume Theorem 4.1 and Remark 4.4 in \cite{PSV} and are rewritten
here in a particular form tailored for our purposes.
\begin{lemma}\label{l4.3}
Under assumptions h1--h4, h6, the initial-boundary value problem
\eqref{2.2}-\eqref{2.3} admits a unique global classical solution
$u\in\C^{2+\alpha,\frac{2+\alpha}{2}\nu_{1}}(\bar{\Omega}_{T}),$
\[
\|u\|_{\C^{2+\alpha,\frac{2+\alpha}{2}\nu_{1}}(\bar{\Omega}_{T})}+\|\D_{t}^{\nu_2}u\|_{\C^{\alpha,\frac{\nu_{1}\alpha}{2}}(\bar{\Omega}_{T})}
\leq
C^{*}[\|g\|_{\C^{\alpha,\frac{\nu_{1}\alpha}{2}}(\bar{\Omega}_{T})}+\|u_{0}\|_{\C^{2+\alpha}(\bar{\Omega})}+
\|\varphi\|_{\C^{1+\alpha,\frac{1+\alpha}{2}\nu_{1}}(\partial\Omega_{T})}
],
\]
where the positive quantity $C^*$ depends only on the Lebesgue
measure of $\Omega$, $T$ and the corresponding norms of the
coefficients. Besides, for any $T_{0}\in(0,T]$, there hold
\begin{equation*}\label{3.9}
\int_{\Omega}u(x,t)dx\in\C^{\nu_1}([0,T_{0}]),\quad
\D_{t}^{\nu_{i}}\int_{\Omega}u(x,t)dx\in\C^{\alpha\nu_1/2}([0,T_{0}]),\,
i=1,2,\quad \D_{t}\int_{\Omega}u(x,t)dx\bigg|_{t=0}=\c_0,
\end{equation*}
in particular, if $T_0=t^{*}$, then  $\psi\in\C^{\nu_1}([0,t^{*}]),$
$\mathbf{D}_{t}^{\nu_i}\psi\in\C^{\alpha\nu_1/2}([0,t^{*}])$ and
$\mathbf{D}_t\psi|_{t=0}=\c_{0}$.
\end{lemma}

As we wrote above, formula \eqref{3.2*} is proved  in \cite[Theorem
3.1]{HPV}. Therefore, if we find $\nu_2$ via \eqref{3.3*}, then
Lemma \ref{l4.3} will provide the existence of $u$ solving
\eqref{2.2}-\eqref{2.3}. Thus, the verification of the solvability
to IP \eqref{2.2}-\eqref{2.4} will be completed.

To verify \eqref{3.3*}, we will exploit Proposition \ref{p3.1} and
Lemma \ref{l4.1}. Indeed, integrating equation \eqref{2.1} over
$\Omega$ and bearing in mind observation \eqref{2.4}, we arrive at
the equality (after performing the standard technical calculations)
\begin{equation}\label{3.10*}
\D_{t}\psi(t)-a_{0}(t)\psi(t)-(\mathcal{K}*b_{0}\psi)(t)=\int_{\Omega}g(x,t)dx-d(\mathcal{K}*\mathcal{I})(t)-\mathcal{I}(t),
\,t\in[0,t^{*}].
\end{equation}
In the case of the I type FDO,  equality \eqref{3.10*} can be
rewritten as
\begin{equation}\label{3.10}
\D_{t}^{\nu_2}\psi(t)=[\rho_{1}(t)\D_{t}^{\nu_1}\psi(t)-\c(t)]\rho_{2}^{-1}(t),
\end{equation}
while in the case of the II type FDO, we have
\begin{equation}\label{3.11}
\D_{t}^{\nu_{2}}(\rho_{2}(t)\psi(t))=\D_{t}^{\nu_{1}}(\rho_{1}(t)\psi(t))-\c(t).
\end{equation}
It is worth noting that in \eqref{3.10}, we exploit the nonvanishing
$\rho_{2}(t)$ if  $t\in[0,t^{*}].$

At this point, we aim to reduce these equalities to \eqref{3.2} with
$\theta=\nu_{1}-\nu_2.$ To this end, we examine the  case of the I
and the II type FDO, separately.

\noindent$\bullet$ Bearing in mind the smoothness of $\psi(t)$ and
collecting Lemma \ref{l4.3} with Proposition \ref{p3.1}, we rewrite
\eqref{3.10} in the form
\begin{equation}\label{3.12}
(\omega_{\nu_{1}-\nu_{2}}*\D_{t}^{\nu_{1}}\psi)(t)=\rho_{1}(t)\rho_{2}^{-1}(t)\D_{t}^{\nu_{1}}\psi(t)-\rho^{-1}_{2}(t)\c(t).
\end{equation}
It is apparent that this relation boils down with \eqref{3.2}, where
we set
\[
v(t)=\D_{t}^{\nu_{1}}\psi(t),\quad
a(t)=\rho_{1}(t)\rho_{2}^{-1}(t),\quad
f_{0}(t)=-\c(t)\rho_{2}^{-1}(t).
\]
Obviously, assumptions h3--h5 and nonvanishing $\rho_{2}(t)$ if
$t\in[0,t^{*}]$ provide the following regularity
\[
\frac{\c}{\rho_{2}},\,\D_{t}^{\nu_{1}}\psi,\,\frac{\rho_{1}}{\rho_{2}}\in\C([0,t^{*}]).
\]
Hence, in order to apply Lemma \ref{l4.1} to \eqref{3.12}, we are
left to verify condition \eqref{3.3}, which in the considered case
reads as
\begin{equation}\label{3.13}
\bigg(\frac{\rho_{1}(0)}{\rho_{2}(0)}+\frac{1}{n^{*}}\bigg)\D_{t}^{\nu_{1}}\psi(0)-\frac{\c_0}{\rho_{2}(0)}\neq
0
\end{equation}
for some $n^{*}\in\mathbb{N}$.

Collecting  \eqref{3.9} with the assumptions of Theorem \ref{t3.1}
arrives at the inequalities
\[
\D_{t}\psi|_{t=0}=\c_{0}\neq 0,
\]
which tell us that  two options occur:

\noindent(i) either $\D_{t}^{\nu_{1}}\psi(0)=0$ but $\c_{0}\neq 0;$

\noindent(ii) or $\D_{t}^{\nu_{1}}\psi(0)\neq 0$ and $\c_{0}\neq 0.$

Clearly, in the case of (i), inequality \eqref{3.13} is fulfilled
for any $n^{*}\in\mathbb{N}.$ Coming to the second option,
inequality \eqref{3.13} holds for any $n^{*}\in\mathbb{N}$
satisfying the relation
\[
\frac{1}{n^{*}}\neq
\frac{1}{\rho_{2}(0)}\bigg(\frac{\c_{0}}{\D_{t}^{\nu_{1}}\psi(0)}-\rho_{1}(0)\bigg).
\]
The existence of such $n^{*}$ is provided by (iii) in Remark
\ref{r4.1}.

As a result, all requirements of  Lemma \ref{l4.1} are satisfied
and, hence, applying this claim to \eqref{3.12}, we find $\nu_{2}$
in the form of \eqref{3.3*} if $\D_{t}$ is the I type FDO.

As for  the case of the II type FDO, exploiting assumption h3 along
with Lemma \ref{l4.3} and Proposition \ref{p3.1} and performing
technical computations, we rewrite \eqref{3.11} in the form
 \begin{equation}\label{3.12*}
(\omega_{\nu_{1}-\nu_{2}}*\D_{t}^{\nu_{1}}(\rho_{2}\psi))(t)
=\D_{t}^{\nu_{1}}(\rho_{2}\psi)(t)+[\D_{t}^{\nu_{1}}(\rho_{1}\psi)(t)-\D_{t}^{\nu_{1}}(\rho_{2}\psi)(t)-\c(t)],
\end{equation}
which boils down with \eqref{3.2}, where we put
\[
a(t)=1,\quad v(t)=\D_{t}^{\nu_{1}}(\rho_{2}\psi)(t),\quad
f_{0}(t)=\D_{t}^{\nu_{1}}(\rho_{1}\psi)(t)-\D_{t}^{\nu_{1}}(\rho_{2}\psi)(t)-\c(t).
\]
Obviously, $\c(t)\in\C([0,t^{*}])$.   Then, in order to utilize
Lemma \ref{l4.1} to \eqref{3.12*}, we have to check the following:

\noindent(I) $\D_{t}^{\nu_{1}}(\rho_{2}\psi)$ and
$\D_{t}^{\nu_{1}}(\rho_{1}\psi)$ are continuous in $[0,t^{*}]$;

\noindent(II) there is an integer positive $n^{*}$ such that
\[
\frac{1}{n^{*}}\D_{t}^{\nu_{1}}(\rho_{2}\psi)(0)-\c_{0}+\D_{t}^{\nu_{1}}(\rho_{1}\psi)(0)\neq
0.
\]

\noindent To verify (I), appealing to \cite[Proposition 5.5]{SV}, we
get
\[
D_{t}^{\nu_{1}}(\rho_{i}\psi)(t)=\rho_{i}(t)\D_{t}^{\nu_{1}}\psi(t)+\psi(0)\D_{t}^{\nu_{1}}\rho_{i}(t)
+\frac{\nu_{1}}{\Gamma(1-\nu_{1})}\mathcal{J}_{\nu_{1}}(t;\rho_{i},\psi),\quad
i=1,2,
\]
with
\[
\mathcal{J}_{\nu_{1}}=
\mathcal{J}_{\nu_{1}}(t;\rho_{i},\psi)=\int_{0}^{t}\frac{[\rho_{i}(t)-\rho_{i}(s)]}{(t-s)^{1+\nu_{1}}}[\psi(s)-\psi(0)]ds.
\]
Taking into account the smoothness of $\psi(t)$ and $\rho_{i}(t)$
 and employing \cite[Lemma 5.6]{SV}, we obtain the
following regularity
\[
\mathcal{J}_{\nu_{1}}\in\C^{\nu_{1}}([0,t^{*}])\quad\text{and}\quad
\rho_{i}\D_{t}^{\nu_{1}}\psi,
\D_{t}^{\nu_{1}}\rho_{i}\in\C^{\alpha\nu_{1}/2}([0,t^{*}]),
\]
which in turn ensures
\[
\D_{t}^{\nu_{1}}(\rho_{i}\psi)\in\C^{\alpha\nu_{1}/2}([0,t^{*}]).
\]

\noindent As for the verification of  (II), denoting
\[
R=R(t)=\frac{\rho_{2}(t)}{\rho_{1}(t)} \quad\text{and}\quad
W=W(t)=\rho_{1}(t)\psi(t),
\]
we first define the function
\[
\mathcal{V}(t,n)=\frac{\D_{t}^{\nu_{1}}(RW)(t)}{n}-\c(t)+\D_{t}^{\nu_{1}}W(t)
\]
for $t\in[0,t^{*}]$ and $n\in\mathbb{N}.$ After that, performing the
straightforward calculations and keeping in mind assumptions of
Theorem \ref{t3.1}, we obtain the smoothness:
\begin{equation}\label{3.14}
 R\in\C^{\nu}([0,T]),\quad
\D_{t}^{\nu_{1}}R\in\C^{\alpha/2}([0,T]),\quad R(0)\neq 0, \quad
\D_{t}^{\nu_{1}}W\in\C^{\alpha\nu_{1}/2}([0,t^{*}]).
\end{equation}
Appealing to \cite[Proposition 5.5]{SV}, we rewrite the function
$\mathcal{V}(t,n)$ in a more appropriate form to the further
analysis
\[
\mathcal{V}(t,n)=[n^{-1}R(t)+1]\D_{t}^{\nu_{1}}W-\c(t)+n^{-1}W(0)\D_{t}^{\nu_{1}}R(t)+\frac{\nu_{1}}{n\Gamma(1-\nu_{1})}\mathcal{J}_{\nu_{1}}(t;R,W).
\]
Taking into account \eqref{3.14}, we derive the following estimates:
\begin{align*}
|\mathcal{J}_{\nu_{1}}(t;R,W)|&\leq C
\|R\|_{\C^{1}([0,T])}\|\D_{t}^{\nu_{1}}W\|_{\C([0,t^{*}])}\int_{0}^{t}s^{\nu_{1}}(t-s)^{-\nu_{1}}ds\leq
C t,\\
|\D_{t}^{\nu_{1}}R(t)|&\leq\frac{1}{\Gamma(1-\nu_{1})}\int_{0}^{t}\bigg|\frac{dR(\tau)}{d\tau}\bigg|(t-\tau)^{-\nu_{1}}d\tau\leq
C\|R\|_{\C^{1}([0,T])}t^{1-\nu_{1}}
\end{align*}
for each $t\in[0,t^{*}]$ and any fixed $n\in\mathbb{N}$, which in
turn provide the equalities
\[
\mathcal{J}_{\nu_{1}}(0;R,W)=0\quad\text{and}\quad
\D_{t}^{\nu_{1}}R(0)=0.
\]
Collecting these relations with \eqref{3.14} and bearing in mind the
definition of $\c_{0}$ (see \eqref{2.1}), we compute
\[
\mathcal{V}(0,n)=[1+n^{-1}R(0)]\D_{t}^{\nu_{1}}W(0)-\c_{0}.
\]
Thanks to the inequalities: $\c_{0}\neq 0$ and
$\D_{t}^{\nu_{1}}W(0)=\D_{t}^{\nu_{1}}(\psi\rho_{1})(0),$ two
options occur:

\noindent$\bullet$ either $\c_{0}\neq 0$ and
$\D_{t}^{\nu_{1}}W(0)\neq 0,$

\noindent$\bullet$ or $\c_{0}\neq 0$ but $\D_{t}^{\nu_{1}}W(0)= 0.$

\noindent The second possibility immediately yields $
\mathcal{V}(0,n)=-\c_{0}\neq 0 $ for any $n\in\mathbb{N}$.

\noindent If $\D_{t}^{\nu_{1}}W(0)\neq 0,$ then
$\mathcal{V}(0,n)\neq 0$ if and only if the positive integer $n$
satisfies the inequality
\[
\frac{1}{n}\neq\frac{\c_{0}-\D_{t}^{\nu_{1}}W(0)}{R(0)\D_{t}^{\nu_{1}}W(0)},
\]
which, in the terms of $\psi$ and $\rho_{i},$ $i=1,2,$ reads as
\[
\frac{1}{n}\neq
-\frac{\rho_{1}(0)\D_{t}^{\nu_{2}}(\rho_{2}\psi)(0)}{\rho_{2}(0)\D_{t}^{\nu_1}(\rho_{1}\psi)(0)}.
\]
Clearly, this inequality holds for any $n\in\mathbb{N},$ if the term
$\frac{\D_{t}^{\nu_{2}}(\rho_{2}\psi)(0)}{\rho_{2}(0)\D_{t}^{\nu_1}(\rho_{1}\psi)(0)}$
is nonnegative.
 Otherwise, it is true (see Remark \ref{r4.1})
for any integer positive $n$ satisfying the inequality
\[
n\geq 1+
\bigg[\frac{-R(0)\D_{t}^{\nu_{1}}(\rho_{1}\psi)(0)}{\D_{t}^{\nu_{2}}(\rho_{2}\psi)(0)}\bigg],
\]
where recalling that the symbol $[\cdot]$ denotes the integer part
of a number.

Thus, this completes the verification of  (II) and, accordingly, the
proof of
 \eqref{3.3*}.
Summing up, we have  constructed the triple $(\nu_{1},\nu_{2},u)$
solving IP \eqref{2.1}-\eqref{2.4}. \qed


\subsection{Uniqueness of a solution in
\eqref{2.2}-\eqref{2.4}}\label{s3.3} Here,  we focus on the
uniqueness of the reconstruction of $(\nu_{1},\nu_{2},u)$ by the
additional measurement \eqref{2.4}.
\begin{lemma}\label{l4.4}
Let assumptions of Theorem \ref{t3.1} hold, then inverse problem
\eqref{2.2}-\eqref{2.4} admits no more than one solution
$(\nu_{1},\nu_{2},u)$, where $\nu_1$ and $\nu_2$ are reconstructed
via the additional measurements \eqref{2.4}, while the function
$u\in C^{2+\alpha,\frac{2+\alpha}{2}\nu_{1}}(\bar{\Omega}_{T})$ is a
unique solution to  the direct problem \eqref{2.2}-\eqref{2.3} with
given $\nu_{1}$ and $\nu_{2}$.
\end{lemma}
\begin{proof}
We will exploit the proof by contradiction. Namely, we assume the
existence of two different triples $(\nu_{1},\nu_{2},u)$ and
$(\bar{\nu}_{1},\bar{\nu}_{2},\bar{u})$ which solve
\eqref{2.2}-\eqref{2.4} with the same right-hand sides, coefficients
in the operators and the observation data. Recasting the arguments
leading to \cite[Lemma 4.2]{HPV}, we end up with the equality
$\nu_{1}=\bar{\nu}_{1}$. Therefore, if we show that
 \begin{equation}\label{eq}
\nu_{2}=\bar{\nu}_{2},
 \end{equation}
 then Lemma \ref{l4.3} arrives at the equality $u=\bar{u},$ the
 latter means
 the uniqueness of a solution to \eqref{2.1}-\eqref{2.3}.

Hence, we are left to examine the equality to $\nu_{2}$ and
 $\bar{\nu}_{2}$. Here we provide the detailed proof of \eqref{eq} in the case of
 $\D_{t}$ being the I type FDO. The case of the II type FDO is treated
 with the similar arguments. For simplicity, we put
 $0<\nu_{2}<\bar{\nu}_{2}<1$. Appealing to \eqref{3.12}, we end up
 with the system
 \[
\begin{cases}
(\omega_{\nu_1-\nu_2}*\D_{\tau}^{\nu_{1}}\psi)(\tau)=\frac{\rho_{1}(\tau)}{\rho_{2}(\tau)}\D_{\tau}^{\nu_{1}}\psi(\tau)-\frac{\c(\tau)}{\rho_{2}(\tau)},\\
(\omega_{\nu_1-\bar{\nu}_2}*\D_{\tau}^{\nu_{1}}\psi)(\tau)=\frac{\rho_{1}(\tau)}{\rho_{2}(\tau)}\D_{\tau}^{\nu_{1}}\psi(\tau)-\frac{\c(\tau)}{\rho_{2}(\tau)}
\end{cases}
 \]
 for each $\tau\in[0,t^{*}]$.
At this point, we will exploit the calculations leading to
\eqref{3.6}. Namely, multiplying the first equality by
 $s_{n,\nu_1-\nu_2}(t-\tau)$ and the second equality by
 $s_{n,\nu_1-\bar{\nu}_2}(t-\tau)$ and integrating over $(0,t),$ $\tau<t\leq
 t^{*}$, we have
\[
\begin{cases}
(s_{n,\nu_{1}-\nu_{2}}*\mathcal{U})(t)=(\frac{1}{n}*\D_{t}^{\nu_{1}}\psi)(t),\\
(s_{n,\nu_{1}-\bar{\nu}_{2}}*\mathcal{U})(t)=(\frac{1}{n}*\D_{t}^{\nu_{1}}\psi)(t),
\end{cases}
\quad\text{where}\quad \mathcal{U}=
[\rho_{1}(t)\rho_{2}^{-1}(t)+n^{-1}]\D_{t}^{\nu_{1}}\psi(t)-\frac{\c(t)}{\rho_{2}(t)}.
\]
The last system tells that
$
(s_{n,\nu_{1}-\nu_{2}}*\mathcal{U})(t)=(s_{n,\nu_{1}-\bar{\nu}_{2}}*\mathcal{U})(t)
$
for any $t\in[0,t^{*}]$ and each fixed $n\in\mathbb{N}$. Performing
the change of the variable: $\tau=tz,$ in the integrals, we obtain
\begin{align}\label{3.16}\notag
&t^{\bar{\nu}_{2}-\nu_{2}}\int_{0}^{1}(1-z)^{\nu_{1}-\nu_{2}-1}E_{\nu_{1}-\nu_{2},\nu_{1}-\nu_{2}}(-nt^{\nu_{1}-\nu_{2}}(1-z)^{\nu_{1}-\nu_{2}})
\mathcal{U}(n,tz)dz
\\&=
\int_{0}^{1}(1-z)^{\nu_{1}-\bar{\nu}_{2}-1}E_{\nu_{1}-\bar{\nu}_{2},\nu_{1}-\bar{\nu}_{2}}(-nt^{\nu_{1}-\bar{\nu}_{2}}(1-z)^{\nu_{1}-\bar{\nu}_{2}})
\mathcal{U}(n,tz)dz
\end{align}
for any $t\in[0,t^{*}]$ and each fixed $n\in\mathbb{N}$. Appealing
to the assumptions of Theorem \ref{t3.1} and the properties of the
Mittag-Leffler functions, we end up with the boundedness of the
functions $\mathcal{U}(n,t)$ and
$E_{\nu_{1}-\nu_{2},\nu_{1}-\nu_{2}}(t)$ and
$E_{\nu_{1}-\bar{\nu}_{2},\nu_{1}-\bar{\nu}_{2}}(t)$ for each
$t\in[0,t^{*}]$ and each fixed $n\in\mathbb{N}$. Besides, the
following inequalities hold:
\[
\mathcal{U}(n^{*},0)\neq 0,\quad
E_{\nu_{1}-\nu_{2},\nu_{1}-\nu_{2}}(0)=\frac{1}{\Gamma(\nu_{2})},\quad
E_{\nu_{1}-\bar{\nu}_{2},\nu_{1}-\bar{\nu}_{2}}(0)=\frac{1}{\Gamma(\bar{\nu}_{2})}
\]
with  $n^{*}\in\mathbb{N}$ (see \eqref{3.13}). We recall that the
assumption $\c_{0}\neq 0$ and Remark \ref{r4.1} arrive at the
existence of $n^{*}\in\mathbb{N}$, which provides  the first
inequality in these relations.

\noindent In fine, keeping in mind the inequality
$\bar{\nu}_{2}>\nu_{2}$, we substitute $n=n^{*}$ to \eqref{3.16} and
then pass to the limit there as $t\to 0$. Exploiting Lebesgu\'{e}s
dominated convergence theorem, we conclude that
\[
0=\frac{\Gamma(\bar{\nu_{2}})}{\Gamma(\nu_{2})}\frac{\int_{0}^{1}(1-z)^{\nu_{1}-\bar{\nu}_{2}-1}dz}{\int_{0}^{1}(1-z)^{\nu_{1}-\nu_{2}-1}dz}
=\frac{(\nu_1-\nu_2)\Gamma(\bar{\nu_{2}})}{(\nu_{1}-\bar{\nu}_{2})\Gamma(\nu_{2})}.
\]
This  contradiction may be removed with admitting $
\nu_{2}=\bar{\nu}_{2}, $ which completes the proof of Lemma
\ref{l4.4}.
\end{proof}

\section{Proof of Lemma \ref{l3.2}}\label{s4}

\noindent For simplicity of presentation, we verify this lemma in
the case of $\mathbf{D}_{t}$ having form \eqref{2.2} and, hence,
$\nu_{i}=\nu_2$, the remaining cases are treated with the similar
arguments.
 Setting $u_{1,2}=u_1,$ $u_{2,2}=u_{2}$, $\beta_{1,2}=\beta_1,$
 $\beta_{2,2}=\beta_{2}$ and
\[
U=u_{2}-u_{1} \quad \text{and}\quad \bar{g}_{0}=\begin{cases}
\rho_{2}[\D_{t}^{\beta_{2}}u_{2}-\D_{t}^{\beta_{1}}u_{2}]\qquad\quad
\text{in the case of the I type FDO,}\\
\D_{t}^{\beta_{2}}(\rho_{2}u_{2})-\D_{t}^{\beta_{1}}(\rho_{2}u_{2})
\quad \text{in the case of the II type FDO,}
\end{cases}
\]
and taking into account that $u_1$ and $u_2$ solve
\eqref{2.2}-\eqref{2.3} with $\nu_{2}=\beta_{1}$ and
$\nu_{2}=\beta_{2},$ respectively, we arrive at the initial-boundary
value problem for unknown function $U$
\begin{equation}\label{4.1}
\begin{cases}
\D_{t}U-\mathcal{L}_{1}U-\mathcal{K}*\mathcal{L}_{2}U=\bar{g}_{0}\quad\text{in}\quad\Omega_{T},\\
\mathcal{M}U+(1-d)\mathcal{K}*\mathcal{M}U=0\quad\,\text{ on}\quad
\partial\Omega_{T},\\
U(x,0)=0\qquad\qquad\qquad\qquad \text{ in}\quad \bar{\Omega},
\end{cases}
\end{equation}
where
$$
\mathbf{D}_t=\begin{cases}
\rho_1(t)\mathbf{D}^{\nu_1}_{t}-\rho_{2}(t)\mathbf{D}_{t}^{\beta_1}\qquad\text{in the case of the I type FDO,}\\
\mathbf{D}^{\nu_1}_{t}\rho_1(t)-\mathbf{D}_{t}^{\beta_1}\rho_{2}(t)\qquad\text{in
the case of the II type FDO.}
\end{cases}
$$
Applying Lemma \ref{l4.3} to this problem  tells that \eqref{4.1}
has a unique global classical solution satisfying the bound
\[
\|U\|_{\C^{2+\alpha,\frac{2+\alpha}{2}\nu_{1}}(\bar{\Omega}_{T})}+\|\mathbf{D}_{t}^{\beta_{1}}U\|_{\C^{\alpha,\frac{\nu_{1}\alpha}{2}}(\bar{\Omega}_{T})}\leq
C\|\bar{g}_{0}\|_{\C^{\alpha,\alpha\nu_{1}/2}(\bar{\Omega}_{T})}.
\]
Hence, to complete the proof of Lemma \ref{l3.2}, we have to obtain
the proper estimate of the term
$\|\bar{g}_{0}\|_{\C^{\alpha,\alpha\nu_{1}/2}(\bar{\Omega}_{T})}.$
Indeed, we are left to  show that
\begin{equation}\label{4.2}
\|\bar{g}_{0}\|_{\C^{\alpha,\alpha\nu_{1}/2}(\bar{\Omega}_{T})}\leq
C[\beta_{2}-\beta_{1}][\|u_{0}\|_{\C^{2+\alpha}(\bar{\Omega})}+\|g_{0}\|_{\C^{\alpha,\alpha\nu_{1}/2}(\bar{\Omega}_{T})}
+
\|\varphi\|_{\C^{1+\alpha,\frac{1+\alpha}{2}\nu_{1}}(\partial\Omega_{T})}]
\end{equation}
with $C$ being independent of the difference $\beta_{2}-\beta_{1}.$

At this point, we  verify \eqref{4.2} in the case of $\D_{t}$ being
the I type FDO. Exploiting Proposition \ref{p3.1}, we rewrite
$\bar{g}_{0}$ in the form
\begin{align*}
\bar{g}_{0}&=\rho_{2}(t)([\omega_{\nu_{1}-\beta_{2}}-\omega_{\nu_{1}-\beta_{1}}]*\D_{t}^{\nu_1}u_{2})(t)\equiv
A_{1}+A_{2},\\
A_{1}&=\rho_{2}(t)\bigg(\frac{1}{\Gamma(\nu_{1}-\beta_{2})}-\frac{1}{\Gamma(\nu_{1}-\beta_{1})}\bigg)(t^{\nu_{1}-\beta_{2}-1}*\D_{t}^{\nu_{1}}u_{2})(t),\\
A_{2}&=\frac{\rho_{2}(t)}{\Gamma(\nu_{1}-\beta_{1})}(k*\D_{t}^{\nu_1}u_{2})(t)\quad
\text{with}\quad
k=k(t)=t^{\nu_{1}-\beta_{2}-1}-t^{\nu_{1}-\beta_{1}-1}.
\end{align*}
As for the evaluation of $A_{1},$ performing straightforward
calculations and taking into account Proposition \ref{p3.1},
assumptions h2, h3 and \cite[Proposition 10.1]{KPSV}, we get
\[
\|A_{1}\|_{\C^{\alpha,\alpha\nu_{1}/2}(\bar{\Omega}_{T})}\leq
C\|\rho_{2}\|_{\C^{\nu}([0,T])}\|\D_{t}^{\beta_{2}}u_{2}\|_{\C^{\alpha,\alpha\nu_{1}/2}(\bar{\Omega}_{T})}(\beta_{2}-\beta_{1})
\]
with the positive quantity  $C$ being independent of
$\beta_{2}-\beta_{1}$.

Coming to the term $A_{2},$ thanks to assumption h2, we can utilize
\cite[Proposition 5.1]{HPV} with
\[
\gamma=1-\nu_1+\beta_{2},\quad\bar{\gamma}=1-\nu_1+\beta_{1},\quad
\mathcal{K}_{0}=1,\quad w(x,t)=\D_{t}^{\nu_{1}}u_{2}(x,t),\quad
\beta=\frac{\alpha\nu_{1}}{2},
\]
and then we end up with the estimate
\[
\|A_{2}\|_{\C^{\alpha,\alpha\nu_{1}/2}(\bar{\Omega}_{T})}\leq
C[\beta_{2}-\beta_{1}]\|\rho_{2}\|_{\C^{\nu}([0,T])}\|\D_{t}^{\nu_{1}}u_{2}\|_{\C^{\alpha,\alpha\nu_{1}/2}(\bar{\Omega}_{T})}.
\]
Thus, collecting estimates of $A_{1}$ and $A_{2},$ and applying the
bound to $u_{2}$ dictated by Lemma \ref{l4.3}, we arrive at
\eqref{4.2} in the case of the I type FDO.

If $\D_{t}$ is the II type FDO, we have
\[
\bar{g}_{0}=[\omega_{\nu_{1}-\beta_{2}}-\omega_{\nu_{1}-\beta_{1}}]*\D_{t}^{\nu_{1}}(\rho_{2}u_{2}).
\]
Exploiting the regularity of $\rho_{2}$ and $\D_{t}^{\nu_{1}}u_{2}$
(see Lemma \ref{l4.3}), we employ \cite[Proposition 5.5]{SV} and,
performing technical calculations, obtain the bound
\[
\|\D_{t}^{\nu_{1}}(\rho_{2}u_{2})\|_{\C^{\alpha,\alpha\nu_{1}/2}(\bar{\Omega}_{T})}\leq
C\|\rho_{2}\|_{\C^{\nu}([0,T])}\|\D_{t}^{\nu_{1}}u_{2}\|_{\C^{\alpha,\alpha\nu_{1}/2}(\bar{\Omega}_{T})}.
\]
After that, recasting the arguments leading to \eqref{4.2} in the
case of  $\D_{t}$ being the I type FDO, we reach the desired result,
which completes the proof of Lemma \ref{l3.2}. \qed


\section{Reconstruction of  $(\nu_{1},\nu_{i^{*}},\rho_{i^{*}},u)$ by the Measurement $\psi$
}\label{s5}

\noindent Here, assuming that $\rho_{i^{*}}(t)\equiv \rho_{i^{*}}$
is unknown constant, we propose the approach to recovery not only
the orders $\nu_{1}$ and $\nu_{i^{*}}$ but also the coefficient
$\rho_{i^{*}}$ via the observation data \eqref{2.4}. Since $\c_{0},$
$\c(t)$ and formula \eqref{3.2*} are independent of $\rho_{i^{*}}$,
the order $\nu_{1}$ can be computed by \eqref{3.2*} even if
$\rho_{i^{*}}$ is unknown. As for finding $\nu_{i^{*}},$ and
$\rho_{i^{*}}$, we first focus on their recovery in the case of the
two-term fractional differential operator $\mathbf{D}_{t}$
\eqref{2.2} and then we extend the obtained results to the case of
$\mathbf{D}_{t}$ having form \eqref{3.4*}. To calculate $\nu_{2},$
we apply the result similar to Lemma \ref{l4.1}, where
$\omega_{\theta}$ is replaced by
$\hat{\omega}_{\theta}=b\omega_{\theta}$ with some constant $b.$
\begin{proposition}\label{p5.1}
Let $b\neq 0$ be a real constant, and let
$v,\mathcal{F}_{0}\in\C([0,T])$. If \eqref{3.3} along with
\begin{equation}\label{5.1}
(\hat{\omega}_{\theta}*v)(t)=\mathcal{F}_{0}(t)
\end{equation}
hold for each $t\in[0,T],$ where
$\mathcal{F}_0(t)=a(t)v(t)+f_{0}(t)$, then $\theta$ satisfies
\eqref{b.1}.

\noindent If,  additionally, there is $t_{0}\in(0,T]$ such that $
\mathcal{F}_{0}(t_{0})\neq 0, $ then
\begin{equation}\label{5.1*}
b=\frac{\mathcal{F}_{0}(t_0)}{(\omega_{\theta}*v)(t_{0})}.
\end{equation}
\end{proposition}
\begin{proof}
The first part of this claim is verified with the arguments (with
minor modifications) leading to Lemma \ref{l4.1}. Namely, instead of
\eqref{3.5}
 we employ the identity
\[
b s_{n, \theta}(t)+n(\hat{\omega}_{\theta}*s_{n, \theta})(t)=b
\]
with any $t\in[0,t^{*}]$ and each fixed $n\in\mathbb{N}$. After
that, denoting
\[
\hat{\mathcal{U}}(n,t)=\mathcal{F}_{0}(t)+bv(t)n^{-1},
\]
and recasting  step-by-step the proof of Lemma \ref{l4.1}, we arrive
at \eqref{b.1} if only
\begin{equation}\label{5.0}
\mathcal{F}_{0}(t)+bv(t)n^{-1}\neq 0
\end{equation}
for some $\hat{n}\in\mathbb{N}.$

Clearly, if either $v(0)=0$ or $b=1$, then assumption \eqref{3.3}
ensures \eqref{5.0} with $\hat{n}=n^{*}$. After that, assuming
$b\neq 1$ and $v(0)\neq 0,$ we aim to show that there is some
integer positive $\hat{n}$ for which \eqref{5.0} holds.  Since
$b\neq 0$, relation \eqref{5.0} is equivalent to the inequality
\begin{equation}\label{5.2}
\frac{1}{\hat{n}}\neq -\frac{a(0)v(0)+f_{0}(0)}{b v(0)}.
\end{equation}
It is apparent that, there is at least one $\hat{n}\in\mathbb{N}$
satisfying \eqref{5.2}. Namely, if the right-hand side of this
inequality is nonpositive, then \eqref{5.2} is fulfilled for all
$\hat{n}\in\mathbb{N}.$ Otherwise, selecting
\[
\hat{n}=1+\bigg[\frac{-bv(0)}{a(0)v(0)+f_{0}(0)}\bigg],
\]
we provide the fulfillment of \eqref{5.2}. Here, we again used the
symbol $[\cdot]$ to denote the integer part of a number.

Coming to \eqref{5.1*}, it is a simple consequence of \eqref{5.1}
and the assumption on nonvanishing the right-hand side in
\eqref{5.1} at $t=t_{0}$. That completes the proof of Proposition
\ref{p5.1}.
\end{proof}
At this point, collecting Proposition \ref{p5.1} with arguments of
Section \ref{s3.2} derives formula \eqref{3.3*} to the computation
of $\nu_{2},$ where $\mathcal{F}(t)$ is replaced by
\begin{equation}\label{5.7}
\widetilde{\mathcal{F}}(t)=\begin{cases}
\rho_{1}(t)\D_{t}^{\nu_{1}}\psi(t)-\c(t)\qquad \text{in the case of
the I type FDO,}\\
\D_{t}^{\nu_{1}}(\rho_{1}(t)\psi(t))-\c(t)\quad \text{in the case of
the II type FDO.}
\end{cases}
\end{equation}
Then, exploiting Proposition \ref{p5.1} allows us to look for the
unknown coefficient $\rho_{2}$. To this end, we should rewrite the
requirements in Proposition \ref{p5.1} in the term of given data in
\eqref{2.2}-\eqref{2.4}.

First, we discuss the case of $\D_{t}$ being the I type FDO, i.e.
\[
\D_{t}u(x,t)=\rho_{1}(t)\D_{t}^{\nu_{1}}u(x,t)-\rho_{2}\D_{t}^{\nu_{2}}u(x,t).
\]
Assuming $\rho_{2}\equiv const,$ the arguments of Section \ref{s3.2}
tell us that equality \eqref{3.12} can be rewritten
\begin{equation*}\label{5.6}
\rho_{2}(\omega_{\nu_{1}-\nu_{2}}*\D_{t}^{\nu_{1}}\psi)(t)=\rho_{1}(t)\D_{t}^{\nu_{1}}\psi(t)-\c(t),\quad
t\in[0,t^{*}].
\end{equation*}
Clearly, this equality boils down with \eqref{5.1} where we set
\[
\theta=\nu_{1}-\nu_2,\quad
\hat{\omega}_{\theta}(t)=\rho_{2}\omega_{\nu_{1}-\nu_{2}},\quad
v(t)=\D_{t}^{\nu_{1}}\psi,\quad a(t)=\rho_{1}(t),\quad
f_{0}(t)=-\c(t), \quad T=t^{*}.
\]
If there is $t_{0}\in(0,t^{*}]$ such that
\begin{equation*}\label{5.3}
\rho_{1}(t_{0})\D_{t}^{\nu_{1}}\psi(t_{0})-\c(t_{0})\neq 0,
\end{equation*}
then we can utilize Proposition \ref{p5.1} and find
\begin{equation*}\label{5.4}
\rho_{2}=\frac{\rho_{1}(t_{0})\D_{t}^{\nu_{1}}\psi(t_{0})-\c(t_{0})}{(\omega_{\nu_{1}-\nu_{2}}*\D_{t}^{\nu_{1}}\psi)(t_{0})}.
\end{equation*}
Coming to the case of the II type FDO, we have
\[
\D_{t}\psi(t)=\D_{t}^{\nu_{1}}(\rho_{1}\psi)(t)-\rho_{2}\D_{t}^{\nu_{2}}\psi(t).
\]
Further, recasting the arguments of Section \ref{s3.2} leading to
\eqref{3.12*}, we derive the equality
\[
(\hat{\omega}_{\nu_{1}-\nu_{2}}*\D_{t}^{\nu_{1}}\psi)(t)=\D_{t}^{\nu_{1}}\psi(t)+[\D_{t}^{\nu_{1}}(\rho_{1}\psi)(t)-\D_{t}^{\nu_{1}}\psi(t)-\c(t)]
\]
for each $t\in[0,t^{*}]$, which has the form of \eqref{5.1} with
\[
v(t)=\D_{t}^{\nu_{1}}\psi,\quad\theta=\nu_{1}-\nu_2,\quad
\mu_{2}=\nu_{2},\quad a(t)=1,\quad
f_{0}=\D_{t}^{\nu_{1}}(\rho_{1}\psi)(t)-\D_{t}^{\nu_{1}}\psi(t)-\c(t).
\]
In Section \ref{s3.2}, we have demonstrated that
$\D_{t}^{\nu_{1}}(\rho_{1}\psi)(t),$ $\c(t)\in\C([0,t^{*}])$ and
$\D_{t}^{\nu_{1}}\psi\in\C^{\alpha\nu_{1}/2}([0,t^{*}])$. Thus, if
there exists $t_{0}\in(0,t^{*}]$ such that
\[
\D_{t}^{\nu_{1}}(\rho_{1}\psi)(t_{0})-\c(t_{0})\neq 0,
\]
then Proposition \ref{p5.1} arrives at  the equality
\[
\rho_{2}=\frac{\D_{t}^{\nu_{1}}(\rho_{1}\psi)(t_{0})-\c(t_{0})}{(\omega_{\nu_{1}-\nu_{2}}*\D_{t}^{\nu_{1}}\psi)(t_{0})}.
\]
Finally, substituting parameters $\nu_{1},$ $\nu_{2}$ and $\rho_{2}$
to  \eqref{2.2}-\eqref{2.3}, we  find $u$ via Lemma \ref{l4.3}.
Thus, exploiting the describing above technique, we solve IP
\eqref{2.2}-\eqref{2.4} related with finding
$(\nu_{1},\nu_{2}.\rho_{2},u)$ by additional measurement
\eqref{2.4}. In summary, we claim the following.
\begin{theorem}\label{t5.1}
Let $\nu_{1},\nu_{2}$ and $\rho_{2}\equiv const.\neq 0$ be unknown
parameters in \eqref{2.2}, and let assumptions of Theorem \ref{t3.1}
hold. Then, $\nu_{1}$ and $\nu_{2}$ are computed  via \eqref{3.2*}
and \eqref{3.3*} with $\mathcal{F}(t)=\widetilde{\mathcal{F}}(t)$
given by \eqref{5.7}.

\noindent If, in addition, there exists  $t_{0}\in(0,t^{*}]$ such
that $ \widetilde{\mathcal{F}}(t_{0})\neq 0,$  then
\begin{equation}\label{rho}
\rho_{2}=\frac{\widetilde{\mathcal{F}}(t_{0})}{(\omega_{\nu_{1}-\nu_{2}}*\D_{t}^{\nu_{1}}\psi)(t_{0})}
\end{equation}
and the function
$u\in\C^{2+\alpha,\frac{2+\alpha}{2}\nu_{1}}(\bar{\Omega}_{T})$
solves the problem \eqref{2.2}-\eqref{2.3}.
\end{theorem}
The next result deals with the unique solution of this IP.
\begin{lemma}\label{l5.1}
Let $(\nu_{1},\nu_{2},\rho_{2})$ be unknown parameters in
\eqref{2.2} with  $\rho_{2}$ being a  constant. Moreover, we assume
that assumptions of Theorem \ref{t5.1} hold. Then IP
\eqref{2.2}-\eqref{2.4} related with finding
$(\nu_{1},\nu_{2},\rho_{2}, u)$ by the measurement \eqref{2.4}
admits no more than one solution.
\end{lemma}
\begin{proof}
 In virtue of Theorem \ref{t3.1}, we are left to show
impossibility two different constants $\rho_{2}$ and
$\bar{\rho}_{2}$ which provide a solvability of
\eqref{2.2}-\eqref{2.4} with the same given data. To this end, we
again exploit the argument by contradiction. For simplicity, we give
a detailed proof in the case of the I type FDO, the remaining case
is tackled in a similar manner.

Assuming that $\rho_{2}\neq\bar{\rho}_{2}$, then \eqref{5.6} leads
to the identities
\[
\rho_{2}(\omega_{\nu_{1}-\nu_{2}}*\D_{t}^{\nu_{1}}\psi)(t)=\rho_{1}(t)\D_{t}^{\nu_{1}}\psi-\c(t)=
\bar{\rho}_{2}(\omega_{\nu_{1}-\nu_{2}}*\D_{t}^{\nu_{1}}\psi)(t)
\]
for all $t\in[0,t^{*}]$, which in turn yield  the equality
\[
(\rho_{2}-\bar{\rho}_{2})(\omega_{\nu_{1}-\nu_{2}}*\D_{t}^{\nu_{1}}\psi)(t)=0.
\]
Since  $\rho_{2}\neq\bar{\rho}_{2}$ (by the assumption), the last
equality is fulfilled if only
\[
0=(\omega_{\nu_{1}-\nu_{2}}*\D_{t}^{\nu_{1}}\psi)(t)=\D_{t}^{\nu_{2}}\psi(t)
 \quad\text{for all}\quad t\in[0,t^{*}].
\]
To state the last equality we again apply Proposition \ref{p3.1}.
 Finally, appealing to the definition of the Caputo fractional
derivative, we end up with the identity
\[
\psi(t)= const.\quad\text{for all}\quad t\in[0,t^{*}],
\]
which immediately provides the vanishing $\D_{t}^{\nu_{1}}\psi$ for
all $t\in[0,t^{*})$ and, accordingly, $\mathbf{D}_{t}\psi\equiv 0$.
 Collecting the last identity with Lemma \ref{l4.3}, we
conclude that $ \c_{0}=0.$ However, this contradicts to the
assumption on the nonvanishing $\c_{0}$. This contradiction
completes the proof of the uniqueness $(\nu_{1},\nu_{2},\rho_{2},u)$
solving \eqref{2.2}-\eqref{2.4}.
\end{proof}

\textit{Results of Theorem \ref{t5.1} and Lemma \ref{l5.1} allow us
to establish the one-valued solvability of IP concerning with
looking for $(\nu_{1},\nu_{2},\rho_{2},u)$ by
\eqref{2.2}-\eqref{2.4}, if $\rho_2$ is unknown constant.}

Concerning the reconstruction of parameters
$(\nu_1,\nu_{i^{*}},\rho_{i^{*}}),$ $i^{*}\in\{2,3,...,M\}$ (i.e. in
the case of $\mathbf{D}_{t}$ having form \eqref{3.4*}), we set
\[
\widetilde{\mathcal{F}}_{1}(t)=\begin{cases}
\mathfrak{C}(t)-\sum\limits_{j=1,i^{*}\neq
j}^{M}\rho_{j}(t)\mathbf{D}_{t}^{\nu_{j}}\psi(t)\quad\text{in the
case of the I type FDO},
\\
\mathfrak{C}(t)-\sum\limits_{j=1,i^{*}\neq
j}^{M}\mathbf{D}_{t}^{\nu_{j}}(\rho_{j}(t)\psi(t))\quad\text{in the
case of the II type FDO},
\end{cases}
\]
 and recast the arguments leading to Theorem \ref{t5.1} and Lemma
 \ref{l5.1} where instead of Theorem \ref{t3.1}, we utilize Theorem
 \ref{t3.2}. Thus,  we end up with the claim.
 \begin{theorem}\label{t5.2}
Let $\nu_1,\nu_{i^{*}}$and $\rho_{i^{*}}\neq 0$ be unknown scalar
parameters in the fractional operator \eqref{3.4*} and let there
exist $t_0\in(0,t^{*}]$ such that
$\widetilde{\mathcal{F}}_{1}(t_0)\neq 0$. Then, under assumptions of
Theorem \ref{t3.2}, the inverse problem \eqref{2.1}--\eqref{3.4*}
admits a unique solution $(\nu_1,\nu_{i^{*}},\rho_{i^{*}},u)$ such
that $\nu_1$ and $\nu_{i^{*}}$ are computed via \eqref{3.2*} and
\eqref{3.3*} where $\mathcal{F}(t)$ is replaced by
$\widetilde{\mathcal{F}}_{1}(t)$, while $\rho_{i^{*}}$ is calculated
via \eqref{rho} with $\widetilde{\mathcal{F}}_{1}(t)$ and
$\nu_{i^{*}}$ in place of $\widetilde{\mathcal{F}}(t)$ and $\nu_2$.
Besides, the function
$u\in\C^{2+\alpha,\frac{2+\alpha}{2}\nu_1}(\bar{\Omega}_{T})$ is a
unique classical solution of the direct problem \eqref{2.1},
\eqref{2.3} and \eqref{3.4*} satisfying the observation \eqref{2.4}.
 \end{theorem}

\section{Influence of Noisy Data on Computation of Orders of
Fractional Derivatives}\label{s7}

\noindent Denoting the noisy measurement and the noise level by
$\psi_{\delta}(t)$ and $\delta,$ respectively, we assume that  the
following error bound
\begin{equation}\label{7.1}
|\psi(t)-\psi_{\delta}(t)|\leq \delta\g(t)
\end{equation}
holds for each $t\in[0,t^{*}]$. Here $\g=\g(t)$ is a nonnegative
function having the form
\begin{equation}\label{7.2}
\g(t)=\begin{cases} o(t^{\nu_{1}})\qquad\qquad
\qquad\qquad\quad\text{ the first-type noise
(FTN)},\\
O(t^{\nu_{1}})\qquad\qquad\qquad\qquad\quad\text{the second-type
noise
(STN)},\\
C_{1}+C_{2}t^{\nu_{1}}|\ln\, t|+C_{3}t^{\nu_{1}-\tilde{\nu}}\quad
\text{the third-type noise (TTN)}
\end{cases}
\end{equation}
with $C_{i}$ being nonnegative constants, $C_{2}+C_{3}>0,$
$\tilde{\nu}\in(0,1).$

It is worth noting that the selection of $\g$ is dictated with the
fact that the observation $\psi(t)$ has done only in the very small
neighborhood of $t=0$. We notice that the similar behavior of
$\g(t)$ is analyzed in our previous works
\cite{KPSV,KPSV2,PSV1,HPV}, where the reconstruction of some
parameters  (by a small-time measurement) in
 subdiffusion equations with the one- and multi-term fractional
differential operator $\D_{t}$ is discussed.

Requirements \eqref{7.1} and \eqref{7.2} tell us that
$\psi(0)=\psi_{\delta}(0)$ in the FTN and STN cases as well in the
TTN case this holds if only $C_{1}=0$. Besides, in the TTN case
there is the following asymptotic representation
\[
t^{-\nu_{1}}\g(t)\to+\infty\qquad\text{as}\qquad t\to 0.
\]
Finally, in the STN case, $\g(t)$ can be rewritten in more suitable
form to the further analysis
\[
\g(t)=C_{4}t^{\nu_{1}}+o(t^{\nu_{1}})
\]
with a positive constant $C_{4}$.

In this section, we aim to evaluate the differences
\[
\Delta_{1}=|\nu_{1}-\nu_{1,\delta}|\qquad\text{and}\qquad
\Delta_{2}=|\nu_{2}-\nu_{2,\delta}|,
\]
where parameters $\nu_{1,\delta}$ and $\nu_{2,\delta}$ are
reconstructed by $\psi_{\delta}$. We notice that the bound of
$\Delta_{i^{*}}=|\nu_{2}-\nu_{2,\delta}|$ (if $M>2$) is estimated
with the  arguments providing $\Delta_2$ and we leave it for
interested readers. Lastly, we mention that the assumptions h2, h5
along with Lemma \ref{l4.3} suggest that $\nu_{1,\delta}$ and
$\nu_{2,\delta}$ make sense only if $ \nu_{2}<\nu_{1,\delta}<1$ and
$ 0<\nu_{2,\delta}<\nu_{1}.
$

We notice that the bound of $\Delta_{1}$ is obtained in \cite[Lemma
6.1]{HPV} and, for the reader's convenience, we recall this claim
(rewritten in our notations) here. To this end,  assuming that
$\nu_{1,\delta}$ is computed via formula \eqref{3.2*} with
$\psi_{\delta}$ instead of $\psi$, that is
\begin{equation}\label{7.3}
\nu_{1,\delta}=\begin{cases} \underset{t\to
0}{\lim}\,\frac{\ln|\psi_{\delta}(t)-\int_{\Omega}u_{0}(x)dx|}{\ln\,
t}\qquad\qquad\qquad \text{in the case of the I type FDO,}\\
\underset{t\to
0}{\lim}\,\frac{\ln|\rho_{1}(t)\psi_{\delta}(t)-\rho_{1}(0)\int_{\Omega}u_{0}(x)dx|}{\ln\,
t}\qquad \text{in the case of the II type FDO,}\\
\end{cases}
\end{equation}
we establish.
\begin{lemma}\label{l7.1}
Let assumptions of Theorem \ref{t3.1} hold, and $\nu_{1,\delta}$ be
calculated via \eqref{7.3}. We require that \eqref{7.1} and
\eqref{7.2} are satisfied with $\delta,\tilde{\nu}\in(0,1),$ $C_1=0$
and the remaining $C_{i}$ being positive. Moreover, in the STN case,
we additionally assume that
\begin{equation}\label{7.4}
|\c_{0}|-\delta C_{4}\rho_{1}(0)\neq 0\qquad\text{and}\qquad
\frac{C_{4}\delta\rho_{1}(0)}{||\c_{0}|-\delta
C_{4}\rho_{1}(0)|}\neq 1.
\end{equation}
Then the following estimates hold
\[
\Delta_{1}=0\quad\text{in the FTN and STN cases and}\quad
\Delta_{1}\leq \tilde{\nu}\qquad\text{in the TTN case}.
\]
\end{lemma}
\begin{remark}\label{r7.0}
It is apparent that a sufficient condition providing the fulfillment
of \eqref{7.4} is the inequality
\[
\frac{C_{4}\delta\rho_{1}(0)}{||\c_{0}|-\delta C_{4}\rho_{1}(0)|}<
1.
\]
\end{remark}
As for estimating $\Delta_{2},$ we notice that (see the technique to
the reconstruction of $\nu_{2}$ in Section \ref{s3})
$\nu_{2,\delta}$ will be dependent not only $\psi_{\delta}$ but also
its corresponding fractional derivatives. This fact dictates the
necessity of the additional requirements on the noisy measurement,
which read as
\begin{equation}\label{7.5}
\psi_{\delta},\,\D_{t}^{\nu_{1}}\psi_{\delta},\,
\D_{t}^{\nu_{2},\delta}\psi_{\delta}\in\C([0,t^{*}]).
\end{equation}
In conclusion, denoting the fractional operator \eqref{2.2} with
$\nu_{2,\delta}$ in place $\nu_{2}$ by $\D_{t,\delta}$ and computing
the left-hand side of
 \eqref{3.10*} (with $\D_{t,\delta}$ instead of $\D_{t}$) on
 $\psi_{\delta}$, we arrive at the equality
\begin{equation}\label{7.6}
\D_{t,\delta}\psi_{\delta}(t)-a_{0}(t)\psi_{\delta}(t)-(\mathcal{K}*b_{0}\psi_{\delta})(t)=F_{\delta}(t)
\end{equation}
for each $t\in[0,t^{*}]$. In these calculations, we used assumption
\eqref{7.5}. As for the function $F_{\delta}(t)$, it has a sense of
the right-hand side  in \eqref{3.10*}.

\noindent After that, setting
\begin{align*}
F(t)&=\int_{\Omega}g(x,t)dx-d(\mathcal{K}*\mathcal{I})(t)-\mathcal{I}(t),\\
\Psi_{\delta}(t)&=\psi(t)-\psi_{\delta}(t)\qquad\text{and}\qquad\Phi_{\delta}(t)=F_{\delta}(t)-F(t),
\end{align*}
and subtracting \eqref{3.10*} from \eqref{7.6}, we arrive at the
following relations for each $t\in[0,t^{*}]:$
\begin{equation}\label{7.7}
\rho_{2}(t)[\D_{t}^{\nu_{2}}\psi(t)-\D_{t}^{\nu_{2,\delta}}\psi_{\delta}(t)]=\Phi_{\delta}(t)-(\mathcal{K}*b_{0}\Psi_{\delta})(t)
-a_{0}(t)\Psi_{\delta}(t)+\rho_{1}(t)\D_{t}^{\nu_{1}}\Psi_{\delta}(t)
\end{equation}
in the case of the I type FDO, and
\begin{equation}\label{7.8}
[\D_{t}^{\nu_{2}}(\rho_{2}\psi)(t)-\D_{t}^{\nu_{2,\delta}}(\rho_{2}\psi_{\delta})(t)]=\Phi_{\delta}(t)-(\mathcal{K}*b_{0}\Psi_{\delta})(t)
-a_{0}(t)\Psi_{\delta}(t)+\D_{t}^{\nu_{1}}(\rho_{1}\Psi_{\delta})(t)
\end{equation}
in the case of the II type FDO.

At this point, bearing in mind the last equalities, we evaluate
$\Delta_{2}$ in the case of the I and the II type FDO, separately.
Further in this and next sections, we denote $x^{*}$ the minimal
point of $\Gamma-$function if $x\geq 0$, i.e. $
\Gamma(1+x^{*})=\underset{x\geq 0}{\min}\Gamma(x),$ $x^{*}\approx
0.4616.$

\subsection{The bound of $\Delta_{2}$ in the case of the I type
FDO}\label{s7.1}

First, the straightforward calculations provide the following
properties of the function $\Psi_{\delta}$.
\begin{corollary}\label{c7.1}
Let $\theta\in(0,1),$ $\mathcal{K}\in L_{1}(0,T)$ and
$\mathcal{R}=\mathcal{R}(t)\in\C^{1}([0,t^{*}])$. We assume that
$\g(t)$ given by \eqref{7.2} is continuous in $[0,t^{*}]$. Then,
under  \eqref{7.1} and \eqref{7.5}, the following inequalities hold
for each $t\in[0,t^{*}]:$

\noindent(i)$$ |(\omega_{\theta}*\mathcal{R}\Psi_{\delta})(t)|\leq
\delta\|\mathcal{R}\|_{\C([0,t])}
\begin{cases}
\frac{C_{\psi}\Gamma(1+\nu_{1})t^{\theta+\nu_{1}}}{\Gamma(1+\theta+\nu_{1})}\qquad\qquad\qquad\qquad\quad\text{
in
the FTN and STN cases,}\\
\frac{C_{1}t^{\theta}}{\Gamma(1+\theta)}+\frac{(C_{2}+C_{3})\Gamma(1+\nu_{1}-\tilde{\nu})t^{\theta+\nu_{1}-\tilde{\nu}}}{\Gamma(1+\theta+\nu_{1}-\tilde{\nu})}
\qquad\text{in the TTN case},
\end{cases}
$$
where the positive constant $C_{\psi}$ is greater than $C_{4}$;

\noindent(ii) $$ \int_{0}^{t}|\Psi_{\delta}(\tau)|d\tau\leq\delta
\begin{cases}
\frac{C_{\psi}t^{1+\nu_{1}}}{1+\nu_{1}}\qquad\qquad\qquad\qquad\text{in
the FTN
and STN cases,}\\
C_{1}t+\frac{C_{2}+C_{3}}{1+\nu_{1}-\tilde{\nu}}t^{\nu_{1}-\tilde{\nu}+1}\qquad\text{in
the TTN case;}
\end{cases}
$$

\noindent(iii)$$ \|\omega_{\theta}*\Psi_{\delta}\|_{L_{1}(0,t)}\leq
t\delta\begin{cases}
\frac{C_{\psi}\Gamma(1+\nu_{1})t^{\theta+\nu_{1}}}{\Gamma(2+\theta+\nu_{1})}\qquad\qquad\qquad\qquad\qquad\text{
in
the FTN and STN cases,}\\
\frac{C_{1}t^{\theta}}{\Gamma(2+\theta)}+\frac{(C_{2}+C_{3})t^{\theta+\nu_{1}-\tilde{\nu}}\Gamma(1+\nu_{1}-\tilde{\nu})}
{\Gamma(2+\nu_{1}-\tilde{\nu}+\theta)} \qquad \text{in the TTN
case};
\end{cases}
$$

\noindent(iv)$$ \|\mathcal{K}*\Psi_{\delta}\|_{L_{1}(0,t)}\leq
\delta\|\mathcal{K}\|_{L_{1}(0,t)}\begin{cases}
\frac{C_{\psi}t^{1+\nu_{1}}}{1+\nu_{1}}\qquad\qquad\qquad\text{ in
the FTN and STN cases,}\\
C_{1}t+\frac{C_{2}+C_{3}}{1+\nu_{1}-\tilde{\nu}}t^{\nu_{1}-\tilde{\nu}+1}\quad
\text{in the TTN case};
\end{cases}
$$

\noindent(v)\begin{align*}
\bigg|\int_{0}^{t}\mathcal{R}(\tau)\D_{\tau}^{\nu_{1}}\Psi_{\delta}(\tau)d\tau\bigg|&\leq\delta\|\mathcal{R}\|_{\C^{1}([0,t])}\\
& \times
\begin{cases}
C_{\psi}t\Gamma(1+\nu_{1})\big[1+\frac{t}{2}\big]\qquad\qquad\qquad\qquad\text{
in
the FTN and STN cases,}\\
\\
\frac{C_{1}t^{1-\nu_{1}}[1+\frac{t}{2-\nu_{1}}]}{\Gamma(2-\nu_{1})}
+
\frac{(C_{2}+C_{3})t^{1-\tilde{\nu}}[1+\frac{t}{2-\tilde{\nu}}]\Gamma(1+\nu_{1}-\tilde{\nu})}
{\Gamma(2-\tilde{\nu})} \qquad \text{in the TTN case}.
\end{cases}
\end{align*}
\end{corollary}

Next, we introduce the function
\begin{align*}
\c_{1}(t)&=\Big(|\c_{0}|-\frac{t^{\frac{\nu_{1}\alpha}{2}}[\rho_{1}(0)\langle\D_{t}^{\nu_{1}}\psi\rangle_{t,[0,t^{*}]}^{(\nu_{1}\alpha/2)}
+
|\rho_{2}(0)|\langle\D_{t}^{\nu_{2}}\psi\rangle_{t,[0,t^{*}]}^{(\nu_{1}\alpha/2)}
]}{\Gamma(1+x^{*})}
\Big)\\
& \times
(\rho_{1}(0)\Gamma(1+\nu_{1})+|\rho_{2}(0)|t^{\frac{2}{\alpha\nu_1}})^{-1},
\end{align*}
and define the positive magnitudes:
\[
\underline{\rho_{2}}=\underset{[0,t^{*}]}{\min}|\rho_{2}(t)|,\quad
\underline{\nu_{2}}=\min\{\nu_{2},\nu_{2,\delta}\},
\]
and
\begin{align*}
t_{1}&=(|\c_{0}|\Gamma(1+x^{*}))^{\frac{2}{\nu_{1}\alpha}}
[\rho_{1}(0)\langle\D_{t}^{\nu_{1}}\psi\rangle_{t,[0,t^{*}]}^{(\nu_{1}\alpha/2)}+
|\rho_{2}(0)|\langle\D_{t}^{\nu_{2}}\psi\rangle_{t,[0,t^{*}]}^{(\nu_{1}\alpha/2)}]^{-\frac{2}{\alpha\nu_{1}}};\\
t_{2}&=\min\Big\{t^{*},t_{1},\exp\Big\{-\gamma-\sum_{n=1}^{\infty}\tfrac{\nu_{1}}{n(n-\nu_{1})}\Big\},
\Big(\frac{\rho_{1}(0)\Gamma(1+\nu_{1})}{|\rho_2(0)|}\Big)^{\frac{2}{\alpha\nu_{1}}}\Big\},
\end{align*}
where $\gamma\approx 0.577$ is the Euler-Mascheroni constant.
\begin{remark}\label{r7.1}
The straightforward calculations ensure the positivity of
$\c_{1}(t)$ if only $\c_{0}\neq 0$ and $t\in[0,t_{2})$.
\end{remark}
\begin{lemma}\label{l7.2}
Let assumptions of Theorem \ref{t3.1} hold. We assume that
$\psi_{\delta}$ satisfies \eqref{7.1}, \eqref{7.2} and \eqref{7.5},
and the function $F_{\delta}(t)$ is bounded for each
$t\in[0,t^{*}]$. If  $\D_{t}$ is the I type FDO, then the following
estimate holds for each $t\in(0,t_{2}],$
\[
\Delta_{2}\leq2\underset{t\in(0,t_{2})}{\inf}
\frac{t^{\underline{\nu_{2}}-1-\nu_{1}}\Gamma(2+\alpha\nu_1/2)\Phi_{1,\delta}(t)}
{\c_{1}(t)[|\ln\,t|-\gamma-\sum_{n=1}^{\infty}\tfrac{\nu_{1}}{n(n-\nu_{1})}]\Gamma(1+x^{*})},
\]
where
\begin{align*}
\Phi_{1,\delta}(t)&=t\tfrac{\underset{\tau\in[0,t]}{\sup}|\Phi_{\delta}(\tau)|}{\underline{\rho_{2}}}+t
\delta
C_{\psi}\bigg[\frac{t^{\nu_{1}}\|a_{0}\|_{\C([0,t])}}{(1+\nu_{1})\underline{\rho_{2}}}
+\Gamma(1+\nu_{1}) +
\frac{\|\mathcal{K}\|_{L_{1}(0,t)}t^{\nu_{1}}\|b_{0}\|_{\C([0,t])}}{(1+\nu_{1})\underline{\rho_{2}}}
\\
& +
\Big\|\frac{\rho_{1}}{\rho_{2}}\Big\|_{\C^{1}([0,t])}\Big(1+\frac{t}{2}\Big)\Gamma(1+\nu_{1})
\bigg]
\end{align*}
in the case of FTN or STN, while in the TTN case
\begin{align*}
\Phi_{1,\delta}(t)&=t\tfrac{\underset{\tau\in[0,t]}{\sup}|\Phi_{\delta}(\tau)|}{\underline{\rho_{2}}}
+t^{1-\max\{\nu_{1},\tilde{\nu}\}} \delta
\bigg[\frac{t^{\max\{\nu_{1},\tilde{\nu}\}}\|a_{0}\|_{\C([0,t])}}{\underline{\rho_{2}}}
\Big(C_{1}+\frac{C_{2}+C_{3}}{1+\nu_{1}-\tilde{\nu}}t^{\nu_{1}-\tilde{\nu}}\Big)
\\
& +\Big \|\frac{\rho_{1}}{\rho_{2}}\Big\|_{\C^{1}([0,t])}
t^{\max\{\nu_{1},\tilde{\nu}\}-\nu_{1}}
\Big(\frac{C_{1}}{\Gamma(2-\nu_{1})}\Big[1+\frac{t}{2-\nu_{1}}\Big]
+\frac{(C_{2}+C_{3})\Gamma(1+\nu_{1}-\tilde{\nu})}{\Gamma(2-\tilde{\nu})}t^{\nu_{1}-\tilde{\nu}}\Big[1+\frac{t}{2-\tilde{\nu}}\Big]\Big)\\
&+\frac{2C_{1}t^{\max\{\nu_{1},\tilde{\nu}\}-\nu_{1}}}{\Gamma(2-\nu_{1})}
+
\frac{(C_{2}+C_{3})\Gamma(1+\nu_{1}-\tilde{\nu})}{\Gamma(2-\tilde{\nu})}t^{\max\{\nu_{1},\tilde{\nu}\}-\tilde{\nu}}
\\
&
+\frac{\|\mathcal{K}\|_{L_{1}(0,t)}t^{\max\{\nu_{1},\tilde{\nu}\}}\|b_{0}\|_{\C([0,t])}}{\underline{\rho_{2}}}
\Big( C_{1}+\frac{C_{2}+C_{3}}{1+\nu_{1}-\tilde{\nu}}
t^{\nu_{1}-\tilde{\nu}} \Big) \bigg].
\end{align*}
\end{lemma}
\begin{remark}\label{r7.2}
If the following inequalities hold
$$
\underset{\tau\in[0,t]}{\sup}|\Phi_{\delta}(\tau)|=O(\delta)\quad\text{and}\quad
0<\delta^{\alpha^{*}}<t_{2}
$$
with some $\alpha^{*}$ satisfying the relations
$$
 0<\alpha^{*}<\begin{cases}
\min\{1,(\nu_{1}-\underline{\nu_{2}})^{-1}\}\qquad\qquad\qquad\qquad\text{in
the FTN and STN cases},
\\
\min\{1,(\nu_{1}-\nu_{2}+\max\{\nu_{1},\tilde{\nu}\})^{-1}\}\qquad\text{in
the TTN case},
\end{cases}
$$
 then Lemma \ref{l7.2} provides the bound
\[
\Delta_{2}\leq
\begin{cases}
O(\delta^{1-\alpha^{*}(\nu_{1}-\underline{\nu_{2}})})\qquad\qquad\qquad\text{
in
the FTN and STN cases,}\\
 O(\delta^{1-\alpha^{*}(\nu_{1}-\nu_{2}+\max\{\nu_{1},\tilde{\nu}\})})\qquad\text{in the TTN
 case}.
\end{cases}
\]
\end{remark}
\noindent\emph{Proof of Lemma \ref{l7.2}.} Here, we provide the
 proof of this claim in the case of $\c_{0}>0$, the remaining
case is analyzed with the similar arguments. For simplicity, we
assume that $0<\nu_{2,\delta}<\nu_{2}$, that is
$\underline{\nu_{2}}=\nu_{2,\delta}$. Taking into account the
nonvanishing of $\rho_{2}(\tau)$ if $\tau\in[0,t^{*}]$, we rewrite
\eqref{7.7} in the following form
\begin{align*}
\frac{d}{d\tau}([\omega_{1-\nu_{2}}-\omega_{1-\nu_{2,\delta}}]*[\psi-\psi(0)])(\tau)&
=-
\frac{d}{d\tau}(\omega_{1-\nu_{2,\delta}}*[\Psi_{\delta}-\Psi_{\delta}(0)])(\tau)+\frac{\Phi_{\delta}(\tau)}{\rho_{2}(\tau)}
\\
&
-\frac{(\mathcal{K}*b_{0}\Psi_{\delta})(\tau)}{\rho_{2}(\tau)}-\frac{a_{0}(\tau)\Psi_{\delta}(\tau)}{\rho_{2}(\tau)}
+
\frac{\rho_{1}(\tau)\D_{\tau}^{\nu_{1}}\Psi_{\delta}(\tau)}{\rho_{2}(\tau)}.
\end{align*}
After that, integrating over $[0,t]$ (with $0<t\leq t^{*}$) and
bearing in mind the continuity of $\psi$ and $\psi_{\delta}$, we
arrive at the relation
\[
([\omega_{1-\nu_{2}}-\omega_{1-\nu_{2,\delta}}]*[\psi-\psi(0)])(t)=\Phi_{2,\delta}(t)
\]
for any $t\in(0,t^{*}]$, where
\begin{align*}
\Phi_{2,\delta}(t)&=-(\omega_{1-\nu_{2,\delta}}*[\Psi_{\delta}-\Psi_{\delta}(0)])(t)+\int_{0}^{t}\frac{\Phi_{\delta}(\tau)}{\rho_{2}(\tau)}d\tau
-\int_{0}^{t}\frac{(\mathcal{K}*b_{0}\Psi_{\delta})(\tau)}{\rho_{2}(\tau)}d\tau
\\
&
-\int_{0}^{t}\frac{a_{0}(\tau)\Psi_{\delta}(\tau)}{\rho_{2}(\tau)}d\tau
+
\int_{0}^{t}\frac{\rho_{1}(\tau)\D_{\tau}^{\nu_{1}}\Psi_{\delta}(\tau)}{\rho_{2}(\tau)}d\tau.
\end{align*}
At this point, we set
\begin{align*}
\mathcal{S}(t)&=\rho_{1}(0)\Gamma(1+\nu_{1})-\rho_{2}(0)t^{\nu_{1}-\nu_{2}}\Gamma(1+\nu_{2}),\\
\mathcal{S}_{1}(t)&=\c_{0}+\nu_{1}\rho_{1}(0)t^{-\nu_{1}}J_{\nu_{1}}(\psi,t)-\nu_{2}\rho_{2}(0)t^{-\nu_{2}}J_{\nu_{2}}(\psi,t),
\end{align*}
where $J_{\nu_{i}}(\psi,t)$ is defined in Lemma \ref{l4.2}. Then,
utilizing the mean value theorem to the difference
$\omega_{1-\nu_{2}}(\tau)-\omega_{1-\nu_{2,\delta}}(\tau)$ and
applying Lemma \ref{l4.2} to $\psi-\psi(0),$ we end up with the
equality
\begin{equation}\label{7.9}
\Delta_{2}\int_{0}^{t}\tau^{\nu_{1}}\frac{\mathcal{S}_{1}(\tau)}{\mathcal{S}(\tau)}\frac{\partial
\omega_{1-\nu^{*}}}{\partial\nu^{*}}(t-\tau)d\tau=\Phi_{2,\delta}(t)
\end{equation}
for each $t\in(0,t^{*}]$, where $\nu^{*}\in[\nu_{2,\delta},\nu_{2}]$
is a middle point.

To evaluate the left-hand side of \eqref{7.9} for each
$t\in(0,t_{1}]$, utilizing consistently \cite[Corollary 5.2]{SV} and
the easily verified inequality
$$
|\ln\tau|-\gamma-\sum_{n=1}^{\infty}\tfrac{\nu^{*}}{n(n-\nu^{*})}>|\ln
t|-\gamma-\sum_{n=1}^{\infty}\tfrac{\nu_{1}}{n(n-\nu_{1})}>0
$$
for any $\tau\in(0,t)$ with $ t<t_{2}$, we   deduce that
\begin{align}\label{7.10}\notag
&\frac{\partial
\omega_{1-\nu^{*}}}{\partial\nu^{*}}(\tau)=\omega_{1-\nu^{*}}(\tau)[|\ln\tau|-\gamma-\sum_{n=1}^{\infty}\tfrac{\nu^{*}}{n(n-\nu^{*})}]
\\&
\geq \Big(|\ln
t|-\gamma-\sum_{n=1}^{\infty}\tfrac{\nu_{1}}{n(n-\nu_{1})}\Big)
\omega_{1-\nu^{*}}(\tau)>0,
\end{align}
if $0\leq\tau\leq t<t_{2}$.

Keeping in mind assumptions on the coefficients $\rho_{i},$
  the function $\psi$ and value $\c_{0}$, we have
 \begin{equation}\label{7.11}
\frac{\mathcal{S}_{1}(\tau)}{\mathcal{S}(\tau)}\geq\c_{1}(t)>0
 \end{equation}
if $0\leq\tau\leq t<t_{2}$. The last inequality in \eqref{7.11} is
dictated by Remark \ref{r7.1}.

Coming back to equality \eqref{7.9} and taking into account
\eqref{7.10} and \eqref{7.11}, we obtain
\[
0<\c_{1}(t)\Big[|\ln
t|-\gamma-\sum_{n=1}^{\infty}\tfrac{\nu_{1}}{n(n-\nu_{1})}\Big]\int_{0}^{t}\tau^{\nu_{1}}\omega_{1-\nu^{*}}(t-\tau)d\tau\Delta_{2}\leq
|\Phi_{2,\delta}(t)|
\]
for each $t\in(0,t_{2})$. In fine, computing the integral in the
left-hand side and exploiting Corollary \ref{c7.1} to manage the
right-hand side, we end up with the desired  estimate, which
completes the proof of this lemma.
 \qed

\subsection{The estimate of $\Delta_{2}$ in the case of the II type
FDO}\label{s7.2}

To evaluate $\Delta_{2}$ in this case,  we will follow the strategy
employed  in Section \ref{s7.1}. First, we introduce the quantities
\[
\c_{2}=\D_{t}^{\nu_{2}}(\rho_{2}\psi)(0),\quad\alpha_{1}=\min\{\alpha\nu_{1}/2,1-\nu_{1}\},\quad
\alpha_{2}=\begin{cases} \nu_{2}\quad\text{if}\quad\c_{2}\neq 0,\\
\nu_{1}\quad\text{otherwise},
\end{cases}
\]
and  the threshold time
\begin{align*}
\hat{t}_{2}&=\min\Big\{\hat{t}_{1},t^{*},\exp\Big(-\gamma-\sum_{n=1}^{\infty}\tfrac{\nu_{1}}{n(n-\nu_{1})}\Big)\Big\},
\\&\text{where}\\
\hat{t}_{1}&=\begin{cases}
\frac{[|\c_{2}|\Gamma(1+x^{*})]^{\frac{2}{\alpha\nu_{1}}}}
{\Big[\Gamma(1+\alpha\nu_{1}/2)\langle\D_{t}^{\nu_{2}}(\rho_{2}\psi)\rangle_{t,[0,t^{*}]}^{(\alpha\nu_{1}/2)}
\Big]^{\frac{2}{\alpha\nu_{1}}}}
\qquad\qquad\qquad\qquad\text{if}\quad\c_{2}\neq 0,\\
\frac{[|\c_{0}|\Gamma(1+\nu_{1}+\frac{\alpha\nu_{1}}{2})]^{\frac{1}{\alpha_{1}}}}
{\Big[\Gamma(1+\nu_{1})\{\Gamma(1+\frac{\alpha\nu_{1}}{2})\langle\D_{t}^{\nu_{1}}(\rho_{1}\psi)\rangle_{t,[0,t^{*}]}^{(\frac{\alpha\nu_{1}}{2})}
+\rho_{1}(0)|\psi(0)|\|\frac{\rho_{1}}{\rho_{2}}\|_{\C([0,t^{*}])}\|(\frac{\rho_{2}}{\rho_{1}})'\|_{\C([0,t^{*}])}
\Gamma(1+\nu_{1}+\frac{\alpha\nu_{1}}{2})\}
\Big]^{\frac{1}{\alpha_{1}}}}\quad\text{otherwise}.
\end{cases}
\end{align*}

In further analysis, we  also need   the function
\[
\c_{3}(t)=\begin{cases}
|\c_{2}|-\frac{t^{\alpha\nu_{1}/2}\Gamma(1+\alpha\nu_{1}/2)}{\Gamma(1+x^{*})}
\langle\D_{t}^{\nu_{2}}(\rho_{2}\psi)\rangle_{t,[0,t^{*}]}^{(\alpha\nu_{1}/2)}\qquad\qquad\text{if}\quad\c_{2}\neq 0,\\
\\
\Big(\underset{[0,t]}{\min}|\tfrac{\rho_{2}}{\rho_{1}}|\Big)\bigg\{\frac{|\c_{0}|}{\Gamma(1+\nu_{1})}-
t^{\alpha_{1}}\Big[\frac{\Gamma(1+\frac{\alpha\nu_{1}}{2})\langle\D_{t}^{\nu_{1}}(\rho_{1}\psi)\rangle_{t,[0,t^{*}]}^{(\frac{\alpha\nu_{1}}{2})}}
{\Gamma(1+\nu_{1}+\frac{\alpha\nu_{1}}{2})}\\
+
\|\frac{\rho_{1}}{\rho_{2}}\|_{\C([0,t^{*}])}\|(\frac{\rho_{2}}{\rho_{1}})'\|_{\C([0,t^{*}])}\rho_{1}(0)|\psi(0)|
\Big] \bigg\}\qquad\qquad\quad \text{otherwise}.
\end{cases}
\]
Clearly, if $t\in[0,\hat{t}_{2}]$, then the function $\c_{3}(t)$ is
positive.

\begin{lemma}\label{l7.3}
Let $\D_{t}$ be the II type FDO and assumptions of Lemma \ref{l7.2}
hold. Then the following estimate holds for each
$t\in(0,\hat{t}_{2}]$,
\[
\Delta_{2}\leq 2\underset{t\in(0,\hat{t}_{2})}{\inf}
\frac{t^{\underline{\nu_{2}}-1-\alpha_{2}}\Phi_{3,\delta}(t)}
{\c_{3}(t)[|\ln\,t|-\gamma-\sum_{n=1}^{\infty}\tfrac{\nu_{1}}{n(n-\nu_{1})}]\Gamma(1+x^{*})},
\]
where
\begin{align*}
\Phi_{3,\delta}(t)&=t\underset{\tau\in[0,t]}{\sup}|\Phi_{\delta}(\tau)|+t
\delta
C_{\psi}\bigg[\frac{t^{\nu_{1}}\|a_{0}\|_{\C([0,t])}}{1+\nu_{1}}
+\Gamma(1+\nu_{1}) +
\frac{\|\mathcal{K}\|_{L_{1}(0,t)}t^{\nu_{1}}\|b_{0}\|_{\C([0,t])}}{1+\nu_{1}}
\\
& + \|\rho_{1}\|_{\C([0,t])}\Gamma(1+\nu_{1}) \bigg]
\end{align*}
in the case of FTN or STN, while in the TTN case
\begin{align*}
\Phi_{3,\delta}(t)&=t\underset{\tau\in[0,t]}{\sup}|\Phi_{\delta}(\tau)|
+t^{1-\max\{\nu_{1},\tilde{\nu}\}} \delta
\bigg[t^{\max\{\nu_{1},\tilde{\nu}\}}\|a_{0}\|_{\C([0,t])}
\Big(C_{1}+\frac{C_{2}+C_{3}}{1+\nu_{1}-\tilde{\nu}}t^{\nu_{1}-\tilde{\nu}}\Big)
\\
& +(1+\|\rho_{1}\|_{\C([0,t])})\Big[
\frac{C_{1}t^{\max\{\nu_{1},\tilde{\nu}\}-\nu_{1}}}{\Gamma(2-\nu_{1})}
+\frac{(C_{2}+C_{3})\Gamma(1+\nu_{1}-\tilde{\nu})}{\Gamma(2-\tilde{\nu})}t^{\max\{\nu_{1},\tilde{\nu}\}-\tilde{\nu}}\Big)\\
&+\frac{[1+\rho_{1}(0)]C_{1}t^{\max\{\nu_{1},\tilde{\nu}\}-\nu_{1}}}{\Gamma(2-\nu_{1})}
+\|\mathcal{K}\|_{L_{1}(0,t)}t^{\max\{\nu_{1},\tilde{\nu}\}}\|b_{0}\|_{\C([0,t])}
\Big( C_{1}+\frac{C_{2}+C_{3}}{1+\nu_{1}-\tilde{\nu}}
t^{\nu_{1}-\tilde{\nu}} \Big) \bigg].
\end{align*}
\end{lemma}
\begin{proof}
Here, by analogy with the proof of Lemma \ref{l7.2}, we assume, for
simplicity, the positivity of $\c_{0}$ and
$\nu_{2,\delta}\in(0,\nu_{2})$. Then integrating \eqref{7.8} over
$(0,t)$ (with arbitrary $t\in(0,\hat{t}_{2})$), we obtain
\[
([\omega_{1-\nu_{2}}-\omega_{1-\nu_{2,\delta}}]*[\rho_{2}\psi-\rho_{2}(0)\psi(0)])(t)=\Phi_{4,\delta}(t)
\]
for any $t\in(0,\hat{t}_{2})$, where
\begin{align*}
\Phi_{4,\delta}(t)&=-(\omega_{1-\nu_{2,\delta}}*[\Psi_{\delta}-\Psi_{\delta}(0)])(t)+\int_{0}^{t}\Phi_{\delta}(\tau)d\tau
-\int_{0}^{t}(\mathcal{K}*b_{0}\Psi_{\delta})(\tau)d\tau
\\
& -\int_{0}^{t}a_{0}(\tau)\Psi_{\delta}(\tau)d\tau +
(\omega_{1-\nu_{1}}*[\rho_{1}\Psi_{\delta}-\rho_{1}(0)\Psi_{\delta}(0)])(t).
\end{align*}
To handle the difference
$[\omega_{1-\nu_{2}}-\omega_{1-\nu_{2,\delta}}]$ in the left-hand
side of this equality, we again appeal to the mean-value theorem and
deduce the equality
\begin{equation}\label{7.12}
\Delta_{2}\int_{0}^{t}\frac{\partial\omega_{1-\nu^{*}}}{\partial\nu^{*}}(t-\tau)[\rho_{2}(\tau)\psi(\tau)-\rho_{2}(0)\psi(0)]d\tau
=\Phi_{4,\delta}(t)
\end{equation}
for each $t\in(0,\hat{t}_{2})$.

It is apparent that the right-hand side in this equality is tackled
with Corollary \ref{c7.1} and, besides, the term
$\frac{\partial\omega_{1-\nu^{*}}}{\partial\nu^{*}}$ is managed via
\eqref{7.10} if $t\in[0,\hat{t}_{1})$.  Thus, we are left to obtain
the proper   representation (to the further evaluation) of the
difference $[\rho_{2}(\tau)\psi(\tau)-\rho_{2}(0)\psi(0)].$ On this
route, bearing in mind that $\c_{0}=\D_{t}\psi(0)>0$, two
possibilities occur:

\noindent (i) either $\c_{2}\neq 0,$

\noindent (ii) or $\c_{2}= 0$ but
$\c_{0}=\D_{t}^{\nu_{1}}(\rho_{1}\psi)(0)>0.$

In the case (i), we assume, for simplicity, $\c_{2}>0$ (otherwise we
multiply equality \eqref{7.12} by $-1$) and, exploiting \cite[Lemma
4.1]{KPSV2}, we arrive at the following inequalities for any
$\tau\in(0,t)$ and each $t\in(0,\hat{t}_{2})$,
\begin{align}\label{7.13}\notag
\rho_{2}(\tau)\psi(\tau)-\rho_{2}(0)\psi(0)&=\frac{\c_{2}\tau^{\nu_{2}}}{\Gamma(1+\nu_{2})}+\frac{1}{\Gamma(\nu_{2})}\int_{0}^{\tau}(\tau-s)^{\nu_{2}-1}
[\D_{s}^{\nu_{2}}(\rho_{2}\psi)(s)-\D_{s}^{\nu_{2}}(\rho_{2}\psi)(0)]ds\\
&\geq \tau^{\nu_{2}}\Big[
\frac{\c_{2}}{\Gamma(1+\nu_{2})}-\frac{t^{\alpha\nu_{1}/2}\Gamma(1+\alpha\nu_{1}/2)}{\Gamma(1+\nu_{2}+\alpha\nu_{1}/2)}
\langle\D_{t}^{\nu_{2}}(\rho_{2}\psi)\rangle_{t,[0,\hat{t}_{2}]}^{(\alpha\nu_{1}/2)}
\Big]\\\notag &\geq\tau^{\nu_{2}}\c_{3}(t)>0.
\end{align}

As for the case (ii),   we  set, for simplicity, $\rho_{2}(\tau)>0$
(otherwise, we multiply \eqref{7.12} by $-1$) and get
\[
\rho_{2}(\tau)\psi(\tau)-\rho_{2}(0)\psi(0)=\frac{\rho_{2}(\tau)}{\rho_{1}(\tau)}[\rho_{1}(\tau)\psi(\tau)-\rho_{1}(0)\psi(0)]
+\rho_{1}(0)\psi(0)\Big[\frac{\rho_{2}(\tau)}{\rho_{1}(\tau)}-\frac{\rho_{2}(0)}{\rho_{1}(0)}\Big].
\]
Then, appealing to \cite[Lemma 4.1]{KPSV2} and the mean-value
theorem, we arrive at the representation
\begin{align*}
\rho_{2}(\tau)\psi(\tau)-\rho_{2}(0)\psi(0)&=\frac{\tau^{\nu_{1}}\rho_{2}(\tau)}{\rho_{1}(\tau)}\Big(
\frac{\c_{0}}{\Gamma(1+\nu_{1})}+\frac{\tau^{-\nu_{1}}}{\Gamma(\nu_{1})}\int_{0}^{\tau}(\tau-s)^{\nu_{1}-1}
[\D_{s}^{\nu_{1}}(\rho_{1}\psi)(s)-\D_{s}^{\nu_{1}}(\rho_{1}\psi)(0)]ds
\Big)\\
&+\tau\rho_{1}(0)\psi(0)\frac{d}{d\tau}\frac{\rho_{2}(\tau)}{\rho_{1}(\tau)}\Big|_{\tau=\tau^{*}}
\end{align*}
with the middle point $\tau^{*}\in[0,\tau]$. After that, performing
technical calculations, we end up with the bound
\begin{align*}
\rho_{2}(\tau)\psi(\tau)-\rho_{2}(0)\psi(0)&\geq
\frac{\tau^{\nu_{1}}\rho_{2}(\tau)}{\rho_{1}(\tau)}\Big(\frac{\c_{0}}{\Gamma(1+\nu_{1})}-\frac{t^{\alpha\nu_{1}/2}\Gamma(1+\alpha\nu_{1}/2)}
{\Gamma(1+\nu_{1}+\alpha\nu_{1}/2)}\langle\D_{t}^{\nu_{1}}(\rho_{1}\psi)\rangle_{t,[0,t^{*}]}^{(\alpha\nu_{1}/2)}\\
& -
\|\rho_{1}/\rho_{2}\|_{\C([0,t])}\|(\rho_{2}/\rho_{1})'\|_{\C([0,t])}t^{1-\nu_{1}}\rho_{1}(0)|\psi(0)|\Big)\\
&\geq \tau^{\nu_{1}}\c_{3}(t)>0
\end{align*}
for each $\tau\in(0,t)$ and $t\in(0,\hat{t}_{2})$. Collecting this
bound with \eqref{7.10}, \eqref{7.12} and \eqref{7.13} and bearing
in mind the definition of $\alpha_{2}$, we conclude that
\[
\Delta_{2}\c_{3}(t)\Big[|\ln\,
t|-\gamma-\sum_{n=1}^{\infty}\tfrac{\nu_{1}}{n(n-\nu_{1})}\Big]\int_{0}^{t}\omega_{1-\nu^{*}}(t-\tau)\tau^{\alpha_{2}}d\tau\leq|\Phi_{4,\delta}|.
\]
In fine, collecting the technical calculations with Corollary
\ref{c7.1}, we end up with the desired estimate which completes the
proof of this claim.
\end{proof}
\begin{remark}\label{r7.3}
It is worth noting that the right-hand side of the estimate to
$\Delta_{2}$ established in Lemma \ref{l7.3} contains
$\langle\D_{t}^{\nu_{1}}(\rho_{1}\psi)\rangle_{t,[0,t^{*}]}^{(\alpha\nu_{1}/2)}$
and
$\langle\D_{t}^{\nu_{2}}(\rho_{2}\psi)\rangle_{t,[0,t^{*}]}^{(\alpha\nu_{1}/2)}$.
In virtue of assumptions h3, h5 (with $\nu_{4}=\nu_{1}$ and
$\nu_{3}=\nu_{2}$) and \cite[Lemmas 5.5-5.6]{SV}, these terms are
managed via the bounds:
\begin{align*}
\langle\D_{t}^{\nu_{1}}(\rho_{1}\psi)\rangle_{t,[0,t^{*}]}^{(\alpha\nu_{1}/2)}&\leq
C_5\|\rho_{1}\|_{\C^{\nu}([0,t^{*}])}\|\D_{t}^{\nu_{1}}\psi\|_{\C^{\alpha\nu_{1}/2}([0,t^{*}])},\\
\langle\D_{t}^{\nu_{1}}(\rho_{2}\psi)\rangle_{t,[0,t^{*}]}^{(\alpha\nu_{1}/2)}&\leq
C_6\|\rho_{2}\|_{\C^{\nu}([0,t^{*}])}\|\D_{t}^{\nu_{1}}\psi\|_{\C^{\alpha\nu_{1}/2}([0,t^{*}])},
\end{align*}
where  positive quantities $C_5$ and $C_6$ depend only on $t^{*},$
$\alpha$.
\end{remark}
\begin{remark}\label{r7.4}
The straightforward technical calculations dictate that the estimate
stated in Remark \ref{r7.2} holds in the case of the II type FDO,
too.
\end{remark}


\section{Numerical Regularized  Reconstruction of Scalar Parameters}\label{s8}

\noindent In this section, we discuss numerical algorithms to
compute the parameters $\nu_1,\nu_i^{*}$ and $\rho_i^{*}$ in the
fractional differential operator $\mathbf{D}_{t}$ (given by either
\eqref{2.2} or \eqref{3.4*}) via the  explicit formulas, but we
consider the case of less smooth integral observation than it is
required in Theorems \ref{t3.1}, \ref{t3.2},  \ref{t5.1} and
\ref{t5.2}. Obviously, the measurements having such extra smoothness
in real life is more of an exception than a natural occurrence.
Namely, in practice, the observation data is often obtained  in a
discrete, noise-distorted form. In connection with this,  the
following very natural questions appear.
 \textit{Is it possible to apply the theoretically
justified formulas (see \eqref{3.2*}, \eqref{3.3*}) in the case of
the nonsmooth observation $\psi(t)$? If so, what is the way of their
optimal exploitation, that is the approach  providing reliable
results?}
 In the following subsections, we partially answer on these
questions.


\subsection{Algorithm of a numerical computation}\label{s8.1}
Suppose that we have the integral measurement $\psi(t)$ of the
solution $u(x,t)$ at discrete time moments $t_{k},$
$k=1,2,\ldots,K$, $0<t_{1}<t_{2}<\ldots<t_{K}\leq t^{*}$. We also
assume the presence of a noise $\{\delta_k\}_{k=1}^K$ getting worse
observations
\[
\psi_{\delta,k}=\int_{\Omega}u(x,t_{k})dx+\delta_{k},\quad
k=1,2,\ldots,K.
\]
Initial condition in \eqref{2.3} tells us that
\[
\int_{\Omega}u(x,0)dx=\int_{\Omega}u_{0}(x)dx=\psi_{0}.
\]
Computational formulas \eqref{3.2*} and \eqref{3.3*} contain
continuous-argument limits (i.e.~with respect to continuous variable
$t$). Hence, in order to exploit these formulas in the case of
discrete noisy measurements $\psi_{\delta,k}$, we have to
reconstruct
 approximately the
function $\psi(t)$ from the values $\psi_{\delta,k},$
$k=0,1,\ldots,K$, where we  set $ \psi_{\delta,0} \equiv \psi_{0} .
$ We recall that the functions $\mathcal{F}(t)$ (see Theorems
\ref{t3.1} and \ref{t3.2}) and $\widetilde{\mathcal{F}}(t)$,
$\widetilde{\mathcal{F}}_1(t)$ (see Theorems \ref{t5.1} and
\ref{t5.2}) in formula \eqref{3.3*} contain not only the observation
$\psi(t)$ but also its fractional derivatives. Thus, in order to
exploit this formula to reconstruct either $\nu_2$ (see Theorems
\ref{t3.1} and \ref{t5.1}) or $\nu_{i^{*}}$ (see Theorems \ref{t3.2}
and \ref{t5.2}), we should also compute the corresponding fractional
derivatives of the approximately reconstructed to $\psi(t)$.

Bearing in mind a consistently coupled character of formulas
\eqref{3.2*}, \eqref{3.3*} and \eqref{rho}, we exploit either
 a two-steps algorithm (if, following Theorem
\ref{t3.1} or \ref{t3.2}, we aim to find $(\nu_1,\nu_{2})$ or
$(\nu_1,\nu_{i*})$) or a three-steps algorithm (if we look for
$(\nu_1,\nu_2,\rho_2)$, see Theorems \ref{t5.1} and \ref{t5.2}). The
first stage deals with reconstruction of order $\nu_1$ via
\eqref{3.2*} and an approximate reconstruction of $\psi(t)$. On this
route, we use a similar   technique that was (successfully) utilized
in our previous papers
 \cite[Section 6]{PSV1}) and \cite[Section 8.2]{HPV};  its plainer
 counterpart was also elaborated in our earlier works \cite{KPSV,KPSV2} dealing with simpler IPs for single-term
 fractional subdiffusion equations featuring small-time noisy solution measurements. In
\cite{PSV1}, $\psi(t)$ was an observation of the solution $u$ at the
spatial point $x_{0}$ for small time, i.e. $\psi(t)=u(x_{0},t),$
$t\in[0,t^{*}]$, while in \cite{HPV} the measurement $\psi(t)$ was
defined similar to \eqref{2.4}. Here, for the reader's convenience,
we describe this approach in our notations. On the second stage,
exploiting the reconstructed $\nu_{1}$ and $\psi$ along with formula
\eqref{3.3*}, we reconstruct  of the order $\nu_{2}$. We notice
that,  $\nu_{i^{*}}$ (see Theorem \ref{t3.2}) is computed with the
same reconstruction technique, hence we omit its description here.
As for the third step (if any), we compute unknown constant
coefficient $\rho_{2}$ or $\rho_{i^{*}}$ via formula \eqref{rho}
with $\nu_1, \nu_2$ or $\nu_{i^{*}}$ and $\psi$  have been found at
the previous two steps.

\noindent \textbf{Step 1:} Appealing to Tikhonov regularization
scheme \cite{TG,IJ}, we approximate $\psi(t)$ from  noised data
$\{\psi_{\delta,k}\}_{k=0}^{K}$ by means of a minimizer of a
penalized least square functional
\begin{equation}\label{8.1*}
\sum_{k=0}^{K}[\psi(t_{k})-\psi_{\delta,k}]^{2}+ \sigma
\|\psi\|^{2}_{L_{t^{-a}}^{2}(0,t_{K})}\longrightarrow \min,
\end{equation}
where $\sigma$ is a regularization parameter. It is worth noting
that, the choice of the weighted space $L^{2}_{t^{-a}}$ in this
functional is dictated by the following asymptotic behavior of
$\psi(t)$ for small time moments, $t\leq t^{*}$ which follows from
Lemmas \ref{l4.2} and \ref{l4.3}:
\begin{equation}
\label{asymptotic_t0} \psi(t)=\int_{\Omega}u_{0}(x)dx+O(t^{\nu_{1}})
\end{equation}
Indeed, this behavior suggests that the target function should be
(at least) square integrable on $(0,t_{K}),$ $t_{K}\leq t^{*}$ with
an unbounded weight $t^{-a}$, $a\in(0,1)$.

\noindent As for an approximate minimizer to \eqref{8.1*}, it is
natural to seek it in the  finite-dimensional form
\begin{equation}\label{8.2}
\psi_{\delta}(\zeta,t)=\sum_{j=1}^{\mathfrak{I}}q_{j}t^{\beta_{j}}
+\sum_{j=\mathfrak{I}+1}^{\mathfrak{P}}q_{j}P_{j-\mathfrak{I}-1}^{(0,-a)}(t/t_{K}).
\end{equation}
Here, the shifted Jacobi polynomials
\[
P_{m}^{(0,-a)}(t/t_{K})=\sum_{i=0}^{m}\left(\begin{array}{c}
    m\\
    i
\end{array}\right)
\left(\begin{array}{c}
    m-a\\
   m-i
\end{array}\right)\!(t/t_{K}-1)^{m-i}(t/t_{K})^{i}\quad\text{with}\quad
t\in(0,t_{K})
\]
are an orthogonal system in $L_{t^{-a}}^{2}(0,t_{K})$, and power
functions $t^{\beta_{j}}$ ($j=1,2,\ldots,\mathfrak{I}$) are
incorporated to
 facilitate capturing
 small-time asymptotics (see \eqref{asymptotic_t0}) of the true
problem solution, whereas
$\beta_{1}<\beta_{2}<\ldots<\beta_{\mathfrak{I}}$ are the initial
guesses for the $\nu_{1}$ value, if any. We notice that the choice
of $\beta_{i}$ is user-defined, and in our calculations in Section
\ref{s8.2}, we use the uniform distribution on (0,1), i.e.
$\beta_{i}=\frac{i}{\mathfrak{I}}$.
 As for the unknown coefficients $q_{j}$ in \eqref{8.2},
they are identified from the corresponding system of linear
algebraic equations:
\[
(\mathbb{E}^{T}\mathbb{E}+\sigma\mathbb{H})\mathbf{q}=\mathbb{E}^{T}\bar{\psi_{\varepsilon}},
\]
where we set
\begin{align*}
\mathbf{q}&=(q_{1},...,q_{\mathfrak{P}}),\quad
\bar{\psi}_{\delta}=(\psi_{\delta,0},\psi_{\delta,1},...,\psi_{\delta,K})^{T},\\
\mathbb{E}&=\{E_{ij}\}^{K,\quad\mathfrak{P}}_{i=0,j=1},\quad
E_{ij}=e_{j}(t_{i}),\\
\mathbb{H}&=\{H_{l,m}\}_{l,m=1}^{\mathfrak{P}},\quad
H_{l,m}=\int_{0}^{t_{K}}t^{-a}e_{l}(t)e_{m}(t)dt,\\
e_{l}(t)&=
\begin{cases}
t^{\beta_{l}},\quad l=1,2,..,\mathfrak{I},\\
P_{l-\mathfrak{I}-1}^{(0,-a)}(t/t_{K}),\quad
l=1+\mathfrak{I},...,\mathfrak{P}.
\end{cases}
\end{align*}
Thus,  the technique  written above completes the approximate
recovery of $\psi(t)$ in the form of $\psi_{\delta}(\sigma,t)$.
After that, we are left to compute the limit in formula
\eqref{3.2*}. We recall that, the numerical calculations of such
limits are (generally) an ill-posed problem (see for details
\cite{LP}) which requires the use of a regularization technique.
Obviously, we can approximate the limit in \eqref{3.2*} as
\begin{equation}\label{8.1}
\nu_{1,\delta}(\sigma,\bar{t})=
\begin{cases}
\frac{\ln|\psi_{\delta}(\sigma,\bar{t})-\psi_{0}|}{\ln
\bar{t}}\qquad\qquad\qquad\text{in the case of the I type FDO},\\
\frac{\ln|\rho_{1}(\bar{t})\psi_{\delta}(\sigma,\bar{t})-\rho_{1}(0)\psi_{0}|}{\ln
\bar{t}}\qquad\text{ in the case of the II type FDO},
\end{cases}
\end{equation}
where a point $t=\bar{t}$  is selected sufficiently close to zero
and, hence, this point can be considered also as  a regularization
parameter.

\noindent Summing up, we conclude that the regularized approximation
$\nu_{1,\delta}(\sigma,\bar{t})$ of the order $\nu_1$ needs the
 two regularization parameters $\sigma$ and $\bar{t}$
which have to be chosen appropriately. Since, in reality, the
amplitudes $\delta_{k}$ of the noise perturbations are (generally)
unknown, the one should exploit the so-called noise level-free
regularization parameter choice rules. One of the oldest but the
simplest (in utilization) and still effective strategies of this
kind is the quasi-optimality criterion  \cite{LP,TG}. It is worth
noting that, its successful use in the choice of multiple
regularization parameters (similar to $\sigma$, $\bar{t}$) has been
demonstrated in our previous papers \cite{KPSV,KPSV2,HPV,PSV1} and
its effectiveness in study of various  inverse problems has been
advocated in \cite{BK}. Bearing in mind these arguments along with
strategy to parameter selection for multipenalty regularization
\cite{FNP}, we introduce two geometric sequences of regularization
parameters
\[
\sigma = \sigma_{i}=\sigma_{1}\xi_{1}^{i-1},\quad
i=1,2,\ldots,K_{1},\quad\text{and}\quad
\bar{t}=\bar{t}_{j}=\bar{t}_{1}\xi_{2}^{j-1},\quad
j=1,2,\ldots,K_{2},
\]
with (user-defined) values $\sigma_1$ and $\bar{t}_{1},$ and
$\xi_{1},\xi_{2}\in(0,1)$. The magnitudes
$\nu_{1,\delta}(\sigma_{i},\bar{t}_{j})$ have to be calculated for
such indices $i$ and $j$. After that for each $\bar{t}_{j}$ we then
should seek $\sigma_{i_{j}}\in\{\sigma_{i}\}_{i=1}^{K_{1}}$ such
that
\begin{subequations}
\label{reconstrunction_alg}
\begin{gather}
|\nu_{1,\delta}(\sigma_{i_{j}},\bar{t}_{j})-\nu_{1,\delta}(\sigma_{i_{j}-1},\bar{t}_{j})|=
\min\{|\nu_{1,\delta}(\sigma_{i},\bar{t}_{j})-\nu_{1,\delta}(\sigma_{i-1},\bar{t}_{j})|,\quad
i=2, 3,\ldots,K_{1}\}.\label{reconstrunction_alg_step1}
\intertext{Next, $\bar{t}_{j_{0}}$ is selected from
$\{\bar{t}_{j}\}_{j=1}^{K_{2}}$ such that}
|\nu_{1,\delta}(\sigma_{i_{j_0}},\bar{t}_{j_0})-\nu_{1,\delta}(\sigma_{i_{j_0-1}},\bar{t}_{j_0-1})|
=
\min\{|\nu_{1,\delta}(\sigma_{i_{j}},\bar{t}_{j})-\nu_{1,\delta}(\sigma_{i_{j-1}},\bar{t}_{j-1})|,\quad
j=2,3,\ldots,K_{2}\}.\label{reconstrunction_alg_step2}
\end{gather}
\end{subequations}
At last, $\nu_{1,\delta}(\sigma_{i_{j_{0}}},\bar{t}_{j_{0}})$ (which
is computed via \eqref{8.1} with $\sigma=\sigma_{i_{j_{0}}}$ and
$\bar{t}=\bar{t}_{j_{0}}$, and  will be henceforth simply denoted by
$\bar{\nu}_{1,\delta}$ for brevity) is chosen as the output of the
proposed algorithm, which completes  the  Step 1.

\noindent\textbf{Step 2:} On this stage, we aim to compute a minor
order of a fractional derivative in the fractional differential
operator $\mathbf{D}_{t}$ via formula \eqref{3.3*}, that is
$\nu_{2}$ (Theorem \ref{t3.1} or \ref{t5.1}) or $\nu_{i*}$ (Theorem
\ref{t3.2}). We notice that all these values are calculated with the
same formula with slightly modification of the function
$\mathcal{F}$. Therefore, here we restrict  ourself to describing
the computational strategy for $\nu_{2}$ given by Theorem
\ref{t3.1}, the remaining orders are computed with the similar
approach.
 To this end, we propose two different strategies and in
the following section, by means of numerical tests, we will compare
these techniques.

\noindent\textit{The First Strategy:} This technique is a
straightforward computation via formula \eqref{3.3*}, where $\nu_1$
and $\psi$ are replaced by $\bar{\nu}_{1,\delta}$ and
$\psi_{\delta}(\sigma,t)$ (reconstructed with Step 1).
 Namely, substituting $\bar{\nu}_{1,\delta},$ $\psi_{\delta}(\sigma,t)$ , $\bar{\sigma}=\sigma_{i_{j_{0}}}$
and $\bar{t}=\bar{t}_{j_{0}}$  in  \eqref{3.1} and \eqref{3.3*}, we
approximate $\nu_2$ via
\begin{equation}\label{8.5}
\bar{\nu}_{2,\delta}=\nu_{2,\delta}(\bar{\sigma},\bar{t})=\bar{\nu}_{1,\delta}
-\log_{\lambda}\bigg|\frac{\mathcal{F}_{\delta}(\bar{\sigma},\bar{t}\lambda)}{\mathcal{F}_{\delta}
(\bar{\sigma},\bar{t})}\bigg|
\end{equation}
with $\lambda\in(0,1)$ (selected by a user) and
\begin{equation}\label{8.3}
\mathcal{F}_{\delta}(\sigma,t)=\begin{cases}
\rho_{2}^{-1}(t)[\rho_{1}(t)
\mathbf{D}_{t}^{\bar{\nu}_{1,\delta}}\psi_{\delta}(\sigma,t)-\c_{\delta}(\sigma,t)]\quad
\text{the I type FDO},\\
\mathbf{D}_{t}^{\bar{\nu}_{1,\delta}}(\rho_{1}(t)\psi_{\delta}(\sigma,t))-\c_{\delta}(\sigma,t)\qquad\qquad
\text{the II type FDO},
\end{cases}
\end{equation}
\[
\c_{\delta}(\sigma,t)=\int_{\Omega}g(x,t)dx-d(\mathcal{K}*\mathcal{I})(t)
+a_{0}(t)\psi_{\delta}(\sigma,t)+(\mathcal{K}*b_{0}\psi_{\delta})(\sigma,t)-\mathcal{I}(t).
\]
Here, appealing to explicit form of the minimizer
$\psi_{\delta}(\sigma,t)$ (see \eqref{8.3}), we calculate
analytically
\begin{equation*}
\mathbf{D}_{t}^{\bar{\nu}_{1,\delta}}\psi_{\delta}(\sigma,t)=
\sum_{j=1}^{\mathfrak{I}}q_{j}\mathbf{D}_{t}^{\bar{\nu}_{1,\delta}}t^{\beta_{j}}
+\sum_{j=\mathfrak{I}+1}^{\mathfrak{P}}q_{j}\mathbf{D}_{t}^{\bar{\nu}_{1,\delta}}P_{j-\mathfrak{I}-1}^{(0,-a)}(t/t_{K}).
\end{equation*}
the same takes place in the case of
$\mathbf{D}_{t}^{\bar{\nu}_{1,\delta}}(\rho_{1}(t)\psi_{\delta}(\sigma,t))$
where we use \cite[Proposition 5.5]{SV} to compute a fractional
derivative of the product  $\rho_{1}(t)\psi_{\delta}(\sigma,t)$. As
for integral terms in $\mathfrak{C}_{\delta}(\sigma,t)$, they may be
computed either analytically or numerically. In conclusion, the
couple $(\bar{\nu}_{1,\delta},\bar{\nu}_{2,\delta})$ computed via
\eqref{8.1} and \eqref{8.5}, respectively, is the outcome of the
proposed two-step algorithm  which actually exploits  the
quasi-optimality approach only on the Step 1.

\noindent\textit{The Second Strategy:} Motivated by discussion in
\cite[Section 7]{HPV} (where several parameters were reconstructed
simultaneously), we incorporate here the regularized reconstruction
scheme not only to find $\nu_1$ but also to recover $\nu_2$. The
last means that instead of $\bar{\sigma}, \bar{t}$, we substitute
new regularization parameters $\hat{\sigma}, \hat{t}$ in
\eqref{8.5}, the mentioned  parameters are selected with the
algorithm presented in the Step 1.

At the first stage, we refine the guess $\beta_j$ in the minimizer
\eqref{8.2}. Namely,
 we replace
$\beta_{j}$ by $\hat{\beta}_{j}$ by the following rule:
$\hat{\beta}_{j^{*}}=\bar{\nu}_{1,\delta}$ and the remaining
$\hat{\beta}_{j}$,
$j\in\{1,2,...j^{*}-1,j^{*}+1,...,\mathfrak{I}\},$ are selected in a
small  neighborhood of $\bar{\nu}_{1,\delta}.$ Then, the approximate
minimizer of \eqref{8.1*} is sought in the form
\begin{equation}\label{8.6}
\hat{\psi}_{\delta}(\sigma,t)=\sum_{j=1}^{\mathfrak{I}}\hat{q}_{j}t^{\hat{\beta}_{j}}
+\sum_{j=\mathfrak{I}+1}^{\mathfrak{P}}\hat{q}_{j}P_{j-\mathfrak{I}-1}^{(0,-a)}(t/t_{K}).
\end{equation}
with the unknown coefficients $\hat{q}_{j}$ solving the algebraic
system
\[
(\widehat{\mathbb{E}}^{T}\widehat{\mathbb{E}}+\sigma\widehat{\mathbb{H}})\widehat{\mathbf{q}}=\widehat{\mathbb{E}}^{T}\bar{\psi_{\delta}}
\]
with
\begin{align*}
\widehat{\mathbf{q}}&=(\hat{q}_{1},...,\hat{q}_{\mathfrak{P}}),\quad
\bar{\psi}_{\delta}=(\psi_{\delta,0},\psi_{\delta,1},...,\psi_{\delta,K})^{T},\quad
\widehat{\mathbb{E}}=\{\hat{E}_{ij}\}^{K,\quad\mathfrak{P}}_{i=0,j=1},\quad
\hat{E}_{ij}=\hat{e}_{j}(t_{i}),\\
\widehat{\mathbb{H}}&=\{\hat{H}_{l,m}\}_{l,m=1}^{\mathfrak{P}},\quad
\hat{H}_{l,m}=\int_{0}^{t_{K}}t^{-a}\hat{e}_{l}(t)\hat{e}_{m}(t)dt,\quad
\hat{e}_{l}(t)=
\begin{cases}
t^{\hat{\beta}_{l}},\quad l=1,2,..,\mathfrak{I},\\
P_{l-\mathfrak{I}-1}^{(0,-a)}(t/t_{K}),\quad
l=1+\mathfrak{I},...,\mathfrak{P}.
\end{cases}
\end{align*}
Next, recasting the arguments leading to
\eqref{reconstrunction_alg_step1}--\eqref{reconstrunction_alg_step2},
we look for the regularization parameters
\[
\hat{\sigma} = \hat{\sigma}_{i}=\hat{\sigma}_{1}\xi_{1}^{i-1},\quad
i=1,2,\ldots,\hat{K}_{1},\quad\text{and}\quad
\hat{t}=\hat{t}_{j}=\hat{t}_{1}\xi_{2}^{j-1},\quad
j=1,2,\ldots,\hat{K}_{2},
\]
with (again user-defined) $\hat{\sigma}_1$ and $\hat{t}_{1},$ and
$\xi_{1},\xi_{2}\in(0,1)$. The quantities
$\nu_{2,\delta}(\hat{\sigma}_{i},\hat{t}_{j})$ have to be computed
for such indices $i$ and $j$ via \eqref{8.5} with
$\mathcal{F}_{\delta}$ and $\mathfrak{C}_{\delta}$ chosen in the
form \eqref{8.3} with $\hat{\psi}_{\delta}(\hat{\sigma},\hat{t})$ in
place of $\psi_{\delta}(\sigma,t)$. After that for each
$\hat{t}_{j}$ we then  look for
$\hat{\sigma}_{i_{j}}\in\{\hat{\sigma}_{i}\}_{i=1}^{\hat{K}_{1}}$
such that
\[
|\nu_{2,\delta}(\hat{\sigma}_{i_{j}},\hat{t}_{j})-\nu_{2,\delta}(\hat{\sigma}_{i_{j}-1},\hat{t}_{j})|=
\min\{|\nu_{2,\delta}(\hat{\sigma}_{i},\hat{t}_{j})-\nu_{2,\delta}(\hat{\sigma}_{i-1},\hat{t}_{j})|,\quad
i=2, 3,\ldots,\hat{K}_{1}\},\] then $\hat{t}_{j_{0}}$ is taken from
$\{\hat{t}_{j}\}_{j=1}^{\hat{K}_{2}}$ such that
\[
|\nu_{2,\delta}(\hat{\sigma}_{i_{j_0}},\hat{t}_{j_0})-\nu_{2,\delta}(\hat{\sigma}_{i_{j_0-1}},\hat{t}_{j_0-1})|
=
\min\{|\nu_{2,\delta}(\hat{\sigma}_{i_{j}},\hat{t}_{j})-\nu_{2,\delta}(\hat{\sigma}_{i_{j-1}},\hat{t}_{j-1})|,\quad
j=2,3,\ldots,\hat{K}_{2}\}.\] At last,
$\hat{\nu}_{2,\delta}:=\nu_{2,\delta}(\hat{\sigma}_{i_{j_{0}}},\hat{t}_{j_{0}})$
is the outcome of the Second Strategy  which along with
$\bar{\nu}_{1,\delta}$ recovered via Step 1  complete our
computational two-step algorithm, exploiting the quasi-optimality
approach to recovery  both $\nu_1$ and $\nu_2$.

\noindent\textbf{Step 3:} We remark that in accordance of Theorems
\ref{t5.1} and \ref{t5.2}, this step is performed only if we should
additionally seek for unknown constant coefficient either $\rho_2$
in the two-term $\mathbf{D}_{t}$ or $\rho_{i^{*}}$ in the multi-term
FDO. To this end, we utilize formula \eqref{rho} with the recovered
values to $\nu_1,\nu_2$ or $\nu_{i^{*}}$, $\psi$ being found in the
previous steps. For simplicity, here we restrict ourselves  with
consideration of two-term $\mathbf{D}_{t}$ and finding $\rho_2$. The
case of M-term fractional differential operator and therefore
finding $\rho_{i^{*}}$ are discussed in the same way. We approximate
$\rho_2$ as
\begin{equation}\label{rho.1}
\rho_{2,\delta}=\frac{\widetilde{\mathcal{F}}(t_0)}{(\omega_{\nu_{1,\delta}-\nu_{2,\delta}}*\mathbf{D}_{t}^{\nu_{1,\delta}}\psi_{\delta})(t_0)},
\end{equation}
where $t_{0}\in(0,t^*]$ is selected by a user such that
$\widetilde{\mathcal{F}}_{\delta}(t_0)\neq 0$, and the value of
$\widetilde{\mathcal{F}}_{\delta}$ is computed by means of
\eqref{5.7} with $\nu_{1,\delta},\psi_{\delta}$ and
$\mathfrak{C}_{\delta}$ in place of $\nu_{1},\psi$ and
$\mathfrak{C}$. As for $\nu_{2,\delta}$ and $\psi_{\delta}$ in
\eqref{rho.1}, one can set either
$\nu_{2,\delta}=\bar{\nu}_{2,\delta}$ or
$\nu_{2,\delta}=\hat{\nu}_{2,\delta}$ and, accordingly,  either
$\psi_{\delta}=\psi_{\delta}(\sigma,t)$ or
$\psi_{\delta}=\hat{\psi}_{\delta}(\sigma,t)$. Namely, such
selection is explained by two different strategies exploited in Step
2 of the computational algorithm. Our preliminary observation of
numerical tests given in next section suggests to choose
$\nu_{2,\delta}=\hat{\nu}_{2,\delta}$ in \eqref{rho.1}, while the
selection of $\psi_{\delta}$ does not essentially influence on the
numerical outcomes. Finally, collecting $\rho_{2,\delta}$ with
approximate values of $\nu_1$ and $\nu_2$ from Steps 1--2  finishes
this stage.

In the next subsection, we demonstrate the performance of the
proposed algorithm to reconstruct  $\nu_1,\nu_2$ and $\rho_2$  by
series of numerical examples.



\subsection{Numerical experiments}\label{s8.2}

Here we consider \eqref{2.1}-\eqref{2.4} stated  in $\Omega_{T}$
with $\Omega=(0,1)$ and the terminal time $T=1$:
\begin{equation*}\label{9.1}
\begin{cases} \mathbf{D}_{t}u- u_{xx}-a_{0}(t)u-\mathcal{K}*
u_{xx}-\mathcal{K}*b_{0}u=\sum_{i=1}^{3}g_{i}(x,t)\equiv g(x,t) \quad\text{in}\quad\Omega_{T},\\
u(x,0)=u_{0}(x)\quad\text{in}\quad\bar{\Omega},\qquad \frac{\partial
u}{\partial\mathbf{N}}=0\quad\text{on}\quad\partial\Omega_{T}.
\end{cases}
\end{equation*}
The noisy observation \eqref{2.4} is simulated via relations
\begin{equation*}\label{9.2}
\psi_{\delta,
k}=\int_{\Omega}u(x,t_{k})dx+\delta\mathfrak{G}(t_{k}),\quad
k=1,2,...,K, \, \, K=21,
\end{equation*}
with $\delta=0.04$ and
 the noisy data $\mathfrak{G}$ having  the form
\[
\mathfrak{G}(t)=\begin{cases} t|\ln\, t|\qquad\text{{\bf FTN}
case},\\
t^{\nu_{1}}\qquad\qquad \text{{\bf STN} case},\\
t^{\nu_{1}}|\ln\, t|\qquad\text{{\bf TTN} case}.
\end{cases}
\]
We consider the following uniform distribution of the observation
time moments $t_{k}:$
\begin{equation*}\label{9.3}
t_{k}= k\tau\qquad k=1,2,...,21 \qquad \text{with} \quad
\tau=10^{-4},
\end{equation*}
 while the sequences of regularization
parameters are selected as
\begin{align*}\label{9.4}\notag
&\sigma_{i}=2^{-i},\quad i=1,2,...,60,\quad
\bar{t}_{j}=2^{1-j}t_{20},\quad j=1,2,...,100;\\
&\hat{\sigma}_{i}=2^{1-i},\quad i=1,2,...,15,\quad
\hat{t}_{j}=2^{1-j}t_{20}, \quad j=1,2,...,10.
\end{align*}
Then the approximate minimizer is  chosen  in form \eqref{8.2} or in
the similar one (in the case of the Second Strategy) with
$\mathfrak{I}=3$ and $\mathfrak{P}=9$ (with $a=0.99$  employed in
the examples below). Finally, we notice that we chose $\lambda=0.5$
in \eqref{8.5} and $t_{0}=t_{20}$ in \eqref{rho.1}.

 Example \ref{e.1}  focuses on the finding orders $\nu_{1}$ and $\nu_{2}$ of the fractional derivatives in two-term fractional
operators $\mathbf{D}_{t}$, while Example \ref{e.2} concerns with
the case of the three-term $\mathbf{D}_{t}$, where we look for
$\nu_{1}$ and $\nu_{3}$. Besides, in Example \ref{e.1}, we search
also  unknown coefficient $\rho_{2}$ and, accordingly, we
reconstruct all parameters in the fractional operator
$\mathbf{D}_{t}$. In all examples, the data $u(x,t_{k})$ are
generated by the explicit solutions of the corresponding
initial-boundary value problems. The outcomes of Examples
\ref{e.1}--\ref{e.2} are listed in Tables
\ref{tab:e1.1}--\ref{tab:e2.1}. In Example \ref{e.1}, for the sake
of place, we report  the complete list of numerical simulation
concerning reconstructed orders $\nu_1, \nu_2$ while numerical
results of the recovery of $\rho_2$ is listed only in \textbf{FTN}
case, since the outcomes to remaining noise show the similar
performance. As for Example \ref{e.2}, the proposed algorithm has
exhibited the analogous efficiency and, hence, we give here the
numerical calculations of $\nu_1,\nu_3$ in the \textbf{FTN} case.
\begin{example}
\label{e.1} We consider the equation in \eqref{9.1}  with
\[
\nu_{1}=\nu,\quad \nu_{2}=\nu/2,\quad \rho_{1}(t)=1/2, \quad
\rho_{2}(t)=1/4,\quad a_{0}(t)=2,\quad b_{0}(t)=\frac{1}{30},\quad
\mathcal{K}(t)=1+t,
\]
 and $\D_{t}$ being the I type FDO, that is
$
\mathbf{D}_{t}u=\frac{1}{2}\mathbf{D}_{t}^{\nu}u-\frac{1}{4}\mathbf{D}_{t}^{\nu/2}u.
 $
 The right-hand sides in \eqref{9.1} are given with
\begin{align*}
& u_{0}(x)=x^{2}(1-x)^{2}, \quad g_{2}(x,t)=-2t^{\nu}-2[1+t][1-6x+7x^{2}-2x^{3}+x^{4}],\\
&g_{1}(x,t)=\frac{\Gamma(1+\nu)}{2}-\frac{\Gamma(1+\nu)}{4\Gamma(1+\nu/2)}t^{\nu/2}+
\frac{x^{2}(1-x)^{2}}{2}\Big[\frac{t^{1-\nu}}{\Gamma(2-\nu)}-\frac{t^{1-\nu/2}}{2\Gamma(2-\nu/2)}\Big]
,\\
 &
 g_{3}(x,t)=-\frac{t^{1+\nu}}{30(1+\nu)}-\frac{t^{2+\nu}}{30(1+\nu)(2+\nu)}-[t+t^{2}+t^{3}/6]
 \Big[\frac{x^{2}(1-x)^{2}}{30}+2-12x+12x^{2}\Big].
\end{align*}
\end{example}
\noindent Performing direct calculations, we conclude that $
u(x,t)=x^{2}(1-x)^{2}[1+t]+t^{\nu} $ solves  this initial-boundary
value problem.  In this example, appealing to Theorem \ref{t5.1} and
the three-step algorithm, we reconstruct $\nu_{1},\nu_{2}$ and
$\rho_2$. It is worth noting that, to reconstruct $\rho_2$, we test
two different options to $\nu_2$ in \eqref{rho.1} as well different
approximated functions to $\psi$. The numerical results have
demonstrated that the best outcomes are provided via $\nu_2\approx
\hat{\nu}_{2,\delta}$, while the selection of approximation to
$\psi$ does not influence essentially to numerical results. Thus, in
Table \ref{tab:e1.1}, we list only the best numerical results to
$\rho_2$ in the \textbf{FTN} case.
\begin{table}
  \begin{center}
 \caption{The quantities $\bar{\nu}_{1,\delta},$ $\bar{\nu}_{2,\delta},$ $\hat{\nu}_{2,\delta}$ and $\hat{\rho}_{2,\delta}$ (in the case of \textbf{FTN}) in Example \ref{e.1}}
   \label{tab:e1.1}
  \begin{tabular}{c|c|c|c|c|c|c|c|c|c|c|c}
      \hline
  \,&   \multicolumn{4}{c|}{\textbf{FTN}} & \multicolumn{3}{c|}{\textbf{STN}} & \multicolumn{3}{c|}{\textbf{TTN}} &\,\\
      \hline
    \!$\nu_1$   &  $\bar{\nu}_{1,\delta}$ & $\bar{\nu}_{2,\delta}$ & $\hat{\nu}_{2,\delta}$& $\hat{\rho}_{2,\delta}$  & $\bar{\nu}_{1,\delta}$ & $\bar{\nu}_{2,\delta}$ &
   $\hat{\nu}_{2,\delta}$ & $\bar{\nu}_{1,\delta}$ & $\bar{\nu}_{2,\delta}$ & $\hat{\nu}_{2,\delta}$ &\!$\nu_2$\\
      \hline
      0.1 & 0.0999 & 0.0562 & 0.0499 & 0.2492 & 0.0956 & 0.0627 & 0.0460 & 0.0658 & 0.0567 & 0.0887 & 0.05\\
      0.2 & 0.1999 & 0.1018 & 0.0995 & 0.2501 & 0.1960 & 0.0915 & 0.1012 & 0.1658 & 0.0892 & 0.1810 & 0.10\\
      0.3 & 0.2999 & 01452 & 0.1476 & 0.2543 & 0.2963 & 0.1526 & 0.1544 & 0.2671 & 0.1319 & 0.1973 & 0.15\\
0.4 & 0.3999 & 0.1896 & 0.1879 & 0.2620 & 0.3962 & 0.2080 & 0.2030 & 0.3667 & 0.1731 & 0.2486 & 0.20\\
0.5 & 0.4999 & 0.2504 & 0.2344 & 0.2532 & 0.4975 & 0.0219 & 0.2828 & 0.4672 & 0.2519 & 0.3148 & 0.25 \\
0.6 & 0.5994 & 0.2389 & 0.2919 &  0.2613 & 0.5962 & 0.2064 & 0.3095 & 0.5661 & 0.1781 & 0.3154 & 0.30\\
0.7 & 0.6986 & 0.1099 & 0.3897 & 0.2838 & 0.6961 & 0.1804 & 0.3641 & 0.6661 & 0.1730 & 0.3337 & 0.35\\
0.8 & 0.7954 & 0.3658 & 0.4838 & 0.2339  & 0.7959 & 0.2031 & 0.4100 & 0.7669 & 0.4420 & 0.4244 & 0.40\\
 \hline
 \end{tabular}
  \end{center}
\end{table}
\begin{example}\label{e.2}
We consider \eqref{9.1} in case of the II type FDO with three
fractional derivatives and
\[
M=3,\, \nu_{1}=\nu,\, \nu_{2}=\frac{\nu}{2},\,
\nu_{3}=\frac{\nu}{3}, \,\rho_{1}(t)=1/2,\,
\rho_{2}(t)=-\frac{1}{4},\, \rho_{3}(t)=\frac{1+t^{2}}{4},\,
a_{0}(t)=2,\, b_{0}(t)=0,
\]
 and
\begin{align*}
& \mathcal{K}(t)=t^{-\gamma}\quad\text{with}\quad \gamma\in(0,1),
\quad u_{0}(x)=x^{2}(1-x)^{2},\\
 & g_{1}(x,t)=x^{2}(1-x)^{2}\Big[15\Gamma(1+\nu)-\frac{15}{2}\frac{\Gamma(1+\nu)}{\Gamma(1+\frac{\nu}{2})}t^{\frac{\nu}{2}}
+\frac{t^{2-\frac{\nu}{3}}}{2\Gamma(3-\frac{\nu}{3})}+\frac{15\Gamma(1+\nu)}{2\Gamma(1+\frac{2\nu}{3})}t^{\frac{2\nu}{3}}
+\frac{15\Gamma(3+\nu)}{2\Gamma(3+\frac{2\nu}{3})}t^{2+\frac{2\nu}{3}}
 \Big],\\
&g_{2}(x,t)=-2[1+30t^{\nu}][1-6x+7x^{2}-2x^{3}+x^{4}],\\
&g_{3}(x,t)=-2[1-6x+6x^{2}]\Big[\frac{t^{1-\gamma}}{1-\gamma}+30t^{1-\gamma+\nu}\frac{\Gamma(1-\gamma)\Gamma(1+\nu)}{\Gamma(2+\nu-\gamma)}\Big].
\end{align*}
\end{example}
\noindent The direct calculations arrives at the explicit form of
the solution $ u(x,t)=x^{2}(1-x)^{2}[1+30t^{\nu}] $ to this
initial-boundary value problem. Here, in accordance of Theorem
\ref{t3.2} and the two-step  algorithm, we seek $\nu_{1}$ and
$\nu_{3}$.
\begin{table}
  \begin{center}
 \caption{The quantities $\bar{\nu}_{1,\delta},$ $\bar{\nu}_{3,\delta},$ $\hat{\nu}_{3,\delta}$  in Example \ref{e.2} in the case of \textbf{FTN}}
   \label{tab:e2.1}
  \begin{tabular}{c|c|c|c|c|c|c|c|c}
           \hline
       \!$\nu_1$   & 0.1 & 0.2 & 0.3 & 0.4 & 0.5 & 0.6 & 0.7 & 0.8 \\
      \!$\bar{\nu}_{1,\delta}$ & 0.0999 & 0.1999 & 0.2999 & 0.3999 & 0.4999 & 0.5998 & 0.6996 & 0.7987\\
      \!$\nu_3=\frac{\nu_1}{3}$   & 0.0333 & 0.0667 & 0.1000 & 0.1333 & 0.1667 & 0.2000 & 0.2333 & 0.2667 \\
      \!$\bar{\nu}_{3,\delta}$ & 0.0256 & 0.0637 & 0.1056 & 0.1540 & 0.1664 & 0.1712 & 0.0834 & 0.1672\\
       \!$\hat{\nu}_{3,\delta}$ & 0.0334 & 0.0667 & 0.1006 & 0.1359 & 0.1671 & 0.1964 & 0.1540 & 0.2337\\
 \hline
 \end{tabular}
  \end{center}
\end{table}

\section{Discussion and Conclusion}
\label{s9} \noindent In conclusion, we notice that  our theoretical
results along with the regularized computational algorithm tested by
numerical examples are  the effective analytical and numerical
approach to simultaneously recovery (not only in the theory but also
in practice) of fractional order derivatives and constant
coefficients in the fractional differential operators modeling
subdiffusion processes with memory. In further, this approach may be
incorporated to finding adequate constitutive relations describing
complex dynamical processes in living systems. In particular,
collecting Theorem \ref{t5.1} with \cite[Theorem 2.1]{HPV} and
exploiting the corresponding numerical algorithm, one can completely
identify memory parameters in the subdiffusion equation describing
oxygen distribution through capillaries to surrounding tissues.
Finally, the proposed recovery algorithm along with numerical
examples work well in the case of discrete measurement. This fact
suggests that theoretically justified explicit  formulas
\eqref{3.2*} and \eqref{3.3*} may be adapted to the case of the
(direct) initial-boundary value problems having not only classical
solutions but also strong or weak solutions whose existence is
provided in \cite{V1,V2}. The latter means relaxed data requirements
as outlined in Theorems \ref{t3.1}, \ref{t3.2},  \ref{t5.1} and
\ref{t5.2}, and therefore, this issue may be a further
investigation.


\subsection*{Acknowledgments}
S.P. has been partially supported by COMET Module S3AI managed by
the Austrian Research Promotion Agency FFG; N.V. has been partially
supported by the Simons Foundation, SFI-PD-Ukraine-00017674, Awardee
Initials.



\end{document}